\documentclass[11pt,a4paper,reqno]{amsart}
\usepackage[utf8]{inputenc}
\usepackage[frenchb,english]{babel}
\usepackage[T1]{fontenc}
\usepackage{mathrsfs}
\usepackage{amsfonts,amssymb,amsmath,amsgen,amsthm}
\usepackage{color}
\usepackage{amsbsy}
\usepackage{mathrsfs}
\usepackage{bbm}
\usepackage{titletoc}
\usepackage{enumerate}
\usepackage{graphicx}
\usepackage{caption}
\usepackage{subcaption}
\usepackage[all]{xy}
\usepackage{placeins}
\theoremstyle{plain}
\newtheorem{theorem}{Theorem}[section]

\newtheorem{lemma}[theorem]{Lemma}

\newtheorem{proposition}[theorem]{Proposition}

\newtheorem{remark}[theorem]{Remark}

\catcode`@=11
\def\section{\@startsection{section}{1}%
  \z@{1.5\linespacing\@plus\linespacing}{.5\linespacing}%
  {\normalfont\bfseries\large\centering}}
\catcode`@=12
\def\CC{{\mathbb C}}
\def\RR{{\mathbb R}}
\def\NN{{\mathbb N}}
\def\ZZ{{\mathbb Z}}

\def\({\left(}
\def\){\right)}
\def\<{\left\langle}
\def\>{\right\rangle}

\def\eps{\varepsilon}

\DeclareMathOperator{\RE}{Re}
\DeclareMathOperator{\IM}{Im}
\DeclareMathOperator{\DIV}{div}

\numberwithin{equation}{section}

\newcommand{\be}{\begin{equation}}
\newcommand{\ee}{\end{equation}}
\newcommand{\bea}{\begin{eqnarray}}
\newcommand{\eea}{\end{eqnarray}}
\newcommand{\bee}{\begin{eqnarray*}}
\newcommand{\eee}{\end{eqnarray*}}
\newcommand{\bse}{\begin{subequations}}
\newcommand{\ese}{\end{subequations}}

\def\eps{\varepsilon}

\def\pa{\partial}

\def\pe{{\Psi^{\eps}}}
\def\ae{{A^\eps}}
\def\se{{S^\eps}}

\def\ao{{A^0}}
\def\so{{S^0}}

\def\N{{\mathcal{N}}}

\begin{document}

\title{Uniformly accurate time-splitting methods for the semiclassical linear Schr\"{o}dinger equation}
\author[P. Chartier]{Philippe Chartier}
\address[P. Chartier]{INRIA Rennes, IRMAR and ENS Rennes, IPSO Project Team, Campus de Beaulieu, F-35042 Rennes}
\email{Philippe.Chartier@inria.fr}
\author[L. Le Treust]{Lo\"{i}c Le Treust}
\address[L. Le Treust]{IRMAR, Universit\'e de Rennes 1 and INRIA, IPSO Project}
\email{loic.le-treust@univ-amu.fr}
\author[F. M\'ehats]{Florian M\'ehats}
\address[F. M\'ehats]{IRMAR, Universit\'e de Rennes 1 and INRIA, IPSO Project}
\email{florian.mehats@univ-rennes1.fr}

\thanks{This work was supported by the ANR-FWF Project Lodiquas ANR-11-IS01-0003 and by the ANR project Moonrise ANR-14-CE23-0007-01.
}
\begin{abstract} This article is devoted to the construction of 
numerical methods which remain insensitive to the smallness of the semiclassical parameter for the linear Schr\"odinger equation in the semiclassical limit. We specifically analyse the convergence behavior of the first-order splitting. Our main result is a proof of  uniform accuracy. We illustrate the properties of our methods with simulations.
\end{abstract}

\maketitle


\section{Introduction}
We are concerned here with the uniformly accurate numerical approximation of the solution $\pe: \RR_+\times \RR^d\to \CC$, $d \geq 1$, of the {\em linear} Schrödinger equation in its semiclassical limit  
\be \label{eq:GPElinear}
i\eps \pa_t \pe =  - \frac{\eps^2}{2} \Delta \pe +\mathcal{V}\pe
\ee
where $\mathcal{V}$ is a smooth potential which does not depend on time.
The initial datum is assumed to be of the form
\be\label{initdataLS}
	\pe(0,\cdot) = A_0(\cdot)e^{iS_0(\cdot)/\eps} \quad \mbox{ with } \quad 
	\|A_0\|_{L^2(\RR^d)} = 1.
\ee
Note that the $L^2$-norm, the energy and the momentum of $\pe(0,\cdot)$, namely
\begin{eqnarray}\label{eq:invariants}
&\mbox{Mass: }&\|\pe(t,\cdot)\|_{L^2(\RR^d)}^2,\\
&\mbox{Energy: }&\int_{\RR^d}\left(\eps^2|\nabla\pe(t,x)|^2 + \mathcal{V}|\pe(t,x)|^2\right)dx,\\
&\mbox{Momentum: }&\eps\IM\int_{\RR^d}\overline{\pe(t,x)}\nabla\pe(t,x)dx,
\end{eqnarray}
are all preserved by the flow of 
\eqref{eq:GPElinear}, whenever $\pe(0,\cdot)\in H^1(\RR^d)$. 

Owing to its numerous occurrences in a vast number of domains of applications in physics, equation \eqref{eq:GPElinear} has been widely studied (see for instance \cite{MR1335452,Pazy1986} and the references therein). In the semiclassical regime where the rescaled Planck constant $\eps$ is small, its asymptotic study allows for an appropriate description of the observables of $\pe$ through the laws of hydrodynamics. We refer to \cite{bookRemi} for a detailed presentation of the semiclassical analysis and to \cite{MR2805153} for a review of both theoretical and numerical issues.

Let us also mention that we do not consider the case where the initial datas are Gaussian wave packets for which efficient schemes have already been  developed \cite{MR2247684,faou2009computing,Gradinaru2014}.

%
%
\subsection{Motivation}
Generally speaking, numerical methods for equation \eqref{eq:GPElinear} exhibit an error of size 
$\Delta t^p/\eps^r + \Delta x^q/\eps^s$, 
where $\Delta t$ and $\Delta x$ are the time and space steps and $p,q,r,s$ strictly positive numbers. For time-splitting methods for instance, the error on the wave function behaves like $\Delta x/\eps+\Delta t^p/\eps$ \cite{MR1880116,MR2739463}. Even if we content ourselves with observables\footnote{These authors performed extensive numerical tests in both linear and nonlinear cases \cite{MR1880116,MR2047194}.},
 the error of a splitting method of Bao, Jin, and Markowich \cite{MR1880116} grows like $\Delta x/\eps+\Delta t^p$. Now, 
achieving a fixed accuracy for varying values of $\eps$ requires to keep both ratios $\Delta t / \eps^{r/p}$ and $\Delta x/ \eps^{s/q}$ constant, and  becomes {\em prohibitively costly} when $\eps \to 0$. Our aim, in this article, is thus to develop {\em new} numerical schemes that are {\em Uniformly Accurate (UA)} w.r.t. $\eps$, i.e. whose accuracy does not deteriorate for vanishing $\eps$. In other words, schemes for which $r, s=0$. This seems highly desirable as all available methods with the exception of \cite{BCM2013}, namely finite difference methods \cite{Akrivis1993,DFP1981,KOAD1993,wu1996}, splitting methods \cite{MR1921908,MR2739463,MR3047949,MR2429878,PM1990,MR842641}, asymptotic splitting methods \cite{Bader2014,MR3562216}, relaxation schemes \cite{MR1785963} and symplectic methods \cite{MR967679} fail to be UA.

It is the belief of the authors that, prior to the construction of UA-schemes, it is necessary to reformulate \eqref{eq:GPElinear} as in \cite{BCM2013} and we now describe how this can be done.

\subsection{Reformulation of the problem}
In the spirit of the {\em Wentzel-Kramers-Brillouin (WKB)} techniques, we decompose $\pe$ as the product of a slowly varying amplitude and a fast oscillating factor\footnote{Considering 
the WKB-ansatz \eqref{eq:WKBansatz} transforms the invariants \eqref{eq:invariants} into respectively
	\be 
\|\ae\|_{L^2(\RR^d)}^2, \quad \int_{\RR^d}\left(|\eps\nabla\ae + i\ae\nabla\se|^2+ \mathcal{V}|\ae|^2\right)dx \; \mbox{ and } \; 
\IM\int_{\RR^d}\overline{\ae}\(\eps\nabla\ae + i\ae\nabla\se\)dx.
	\ee
}
\be\label{eq:WKBansatz}
	\pe(t,\cdot) = \ae(t,\cdot)e^{i\se(t,\cdot)/\eps}.
\ee
From this point onwards, various choices are possible, depending on whether $\ae$ is complex or not\footnote{The Madelung transform \cite{madelung} relates the semiclassical limit of \eqref{eq:GPElinear} to hydrodynamic equations 
	\begin{equation*}
		\pe(t,\cdot) = \sqrt{\rho^\eps(t,\cdot)}e^{i\se(t,\cdot)/\eps}
	\end{equation*}
	and amounts to choosing $\ae\in \RR_+$.
	However, this formulation leads to both analytical and numerical difficulties 
	in the presence of vacuum, \textit{i.e.} whenever $\rho^\eps$ vanishes
	\cite{CDS2012,MR2375117}.
}: taking $\ae\in \CC$ leads to the following system \cite{bookRemi} 
\bse\label{eq:SCv1}
	\begin{align}
	&\label{eq:SCv1phase}\pa_t \se + \frac{|\nabla \se|^2}{2}  + \mathcal{V}= 0,\\
	&\label{eq:SCv1amplitude}\pa_t \ae +\nabla \se\cdot \nabla \ae +\frac{\ae}{2}\Delta \se
	 = \frac{i\eps\Delta \ae}{2}
	 \end{align}
\ese
with $ \se(0,\cdot) = S_0(\cdot)$ and $\ae(0,\cdot) = A_0(\cdot)$. Under appropriate smoothness assumptions, $(\ae, \se)  \in \CC \times \RR$ converges when  $\eps \to 0$ to the solution $(\ao, \so)$ of
\bse\label{eq:SCv2}
	\begin{align}
	&\label{eq:SCv2phase}\pa_t \so+ \frac{|\nabla \so|^2}{2}  + \mathcal{V}= 0,\\
	&\label{eq:SCv2amplitude}\pa_t \ao +\nabla \so\cdot \nabla \ao +\frac{\ao}{2}\Delta \so
	 = 0.
	 \end{align}
\ese
Notice that $(\rho, v) = (|\ao|^2,\nabla \so)$ is then solution of the Euler system %
\bse\label{eq:Euler}
	\begin{align}
	&\label{eq:Eulerphase}\pa_t v + v\cdot \nabla v  + \nabla\mathcal{V}= 0,\\
	&\label{eq:Euleramplitude}\pa_t \rho +\DIV(\rho v) = 0.
	 \end{align}
\ese

Now, an important drawback of \eqref{eq:SCv1} stems from the formation of caustics in finite time \cite{bookRemi}:  the solution of \eqref{eq:SCv1} may indeed cease to be smooth even though $\pe$ is globally well-defined for $\eps>0$. In order to obtain global existence for $\eps>0$, Besse, Carles and Méhats \cite{BCM2013} suggested an alternative formulation by introducing an asymptotically-vanishing viscosity term in the eikonal equation \eqref{eq:SCv1phase}. Therein, system \eqref{eq:SCv1} is replaced by
\bse\label{eq:SClinear}
	\begin{align}
	&\label{eq:SClinearphase}\pa_t \se + \frac{|\nabla \se|^2}{2}  + \mathcal{V}= \eps^{2}\Delta \se,\\
	&\label{eq:SClinearamplitude}\pa_t \ae +\nabla \se\cdot \nabla \ae +\frac{\ae}{2}\Delta \se
	 = \frac{i\eps\Delta \ae}{2} -i\eps\ae\Delta\se
	 \end{align}
\ese
where $\se(0,x) = S_0(x)$, $\ae(0,x) = A_0(x)$ and where $x\in \RR^d$.
Let us emphasize that both  \eqref{eq:SCv1} and \eqref{eq:SClinear} are equivalent to \eqref{eq:GPElinear} in the following sense: as long as the solution $(\se,\ae)$ of \eqref{eq:SCv1} (resp. \eqref{eq:SClinear}) is smooth, the function $\pe$ defined by \eqref{eq:WKBansatz} solves \eqref{eq:GPElinear}. The well-posedeness of \eqref{eq:SClinear} and the uniform control of the solutions with respect to $\eps$ are stated in Theorem \ref{thm:existencecontrole} below.

The main advantage of the WKB reformulation \eqref{eq:SClinear} over \eqref{eq:GPElinear} is apparent: the semiclassical parameter $\eps$ does not give rise to singular perturbations\footnote{
The Cole-Hopf transformation \cite[Section $4.4.1$]{evans1998partial}
\begin{eqnarray} \label{rem:cole-hopf}
w^\eps = \exp\(-\frac{\se}{2\eps^{2}}\) - 1
\end{eqnarray}
transforms \eqref{eq:SClinearphase} into $\pa_tw^\eps - \frac{\mathcal{V}}{2\eps^2}(w^\eps+1) = \eps^2\Delta w^\eps$ 
for which the regularizing effect of the viscosity term becomes arguably more apparent.}. 
Hence, it constitutes a good basis for the development of UA schemes (at least prior to the appearance of caustics), as witnessed by the methods introduced later in this paper.  

\subsection{Construction of the schemes}\label{sec:schemeconstr}
First and only (up to our knowledge) UA schemes are based on the formulation \eqref{eq:SClinear} introduced in \cite{BCM2013}. Nevertheless, these schemes are still subject to CFL stability conditions and are of low order in time and space. In this paper, we consider, in lieu of finite differences as in \cite{BCM2013}, time-splitting methods, for they enjoy the following favorable features: 
\begin{enumerate}[(i)]
	\item they do not suffer from stability restrictions on the time step;
	\item they are easy to implement;
	\item they preserve exactly the $L^2$-norm;
	\item \label{pt:nonlinear} they can be adapted to semilinear Schrödinger equations;
	\item  \label{pt:high} they can be composed to attain high-order of convergence in time while remaining spectrally convergent in space.
\end{enumerate}

	Points \eqref{pt:nonlinear} and \eqref{pt:high} will be addressed in a forthcoming work using complex time steps (see \cite{BCCM11-2}), while, in this paper,  we introduce first and second order in time splitting-schemes and concentrate on the numerical analysis of the first-order one for the sake of clarity.

	System \eqref{eq:SClinear} is split into four pieces as follows:
		\smallskip
	
	\noindent
	{\em First flow: }  We denote $\varphi^1_h$  the approximate flow at time $h\in \RR$ of the system 
	\bse\label{eq:SCflot1}
		\begin{align}
		&\label{eq:SCflot1phase}\pa_t S + \frac{|\nabla S|^2}{2}  = 0,\\
		&\label{eq:SCflot1amplitude}\pa_t A +\nabla S\cdot \nabla A +\frac{A}{2}\Delta S
	 	= \frac{i\Delta A}{2}.
		\end{align}
	\ese
	The eikonal equation \eqref{eq:SCflot1phase} is solved by means of the method of characteristics, while equation \eqref{eq:SCflot1amplitude} is dealt with by noticing that $w = A \exp\(iS\)$ satisfies the free Schr\"odinger equation $i\pa_t w = -\frac{1}{2}\Delta w$. 
		\smallskip
	
	\noindent
	{\em Second flow: } We define $\varphi^2_h$ as the exact flow at time $h\in\RR$ of the system
	\bse\label{eq:SCflot2}
		\begin{align}
		&\label{eq:SCflot2phase}\pa_t S = 0,\\
		&\label{eq:SCflot2amplitude}\pa_t A 
	 	= \frac{i\(\eps-1\)\Delta A}{2},
		\end{align}
	\ese
	which is solved in the Fourier space.
		\smallskip
	
	\noindent
	{\em Third flow: }   The third flow $\varphi^3_h$ is defined as the exact flow at time $h\in\RR$ of system
	\bse\label{eq:SCflot3linear}
		\begin{align}
		&\label{eq:SCflot3linearphase}\pa_t S  = -\mathcal{V}, \\
		&\label{eq:SCflot3linearamplitude}\pa_t A  =0.
		\end{align}
	\ese
	\smallskip
	
	\noindent
	{\em Fourth flow: }   The fourth flow $\varphi^4_h$ is defined as the exact flow at time $h\in\RR_+$ of 
	\bse\label{eq:SCflot4}
		\begin{align}
		&\label{eq:SCflot4phase}\pa_t S  = \eps^2\Delta S, \\
		&\label{eq:SCflot4amplitude}\pa_t A  =-i\eps A\Delta S.
		\end{align}
	\ese
	Equation \eqref{eq:SCflot4phase} is solved in Fourier space and the solution of \eqref{eq:SCflot4amplitude} is simply obtained through the formula
$A(h,\cdot) = \exp\(-i\eps^{-1}(S(h,\cdot)-S(0,\cdot))\)A(0,\cdot)$. 
	Notice that $\varphi^4_h$ can thus be viewed as a regularizing flow. \\
	%

	%
	%
	%
	%
	%
	%
	\smallskip
The first-order scheme that we consider for  \eqref{eq:SClinear} is then the concatenation of all previous flows
	\be\label{scheme1}
		 \varphi^1_h\circ\varphi^2_h\circ\varphi^3_h\circ \varphi^4_h
	\ee
	while the second-order scheme is given by
	\be\label{scheme2}
		 \varphi^1_{h/2}\circ\varphi^2_{h/2}\circ\varphi^3_{h/2}\circ \varphi^4_h\circ\varphi^3_{h/2}\circ\varphi^2_{h/2}\circ\varphi^1_{h/2}.
	\ee	

    	\subsection{Main result}
	 The main result of this paper is the following theorem: it  states that $\varphi^1_h\circ\varphi^2_h\circ\varphi^3_h\circ \varphi^4_h$  is uniformly accurate w.r.t. the semi-classical parameter $\eps$. The proper statement of the result uses the norm $\|\cdot\|_{s}$ on the set  $\Sigma_s = H^{s+2}(\RR^d)\times H^s(\RR^d)$ defined 	for $s\geq0$ and $u = (S,A)$ by
	
    	\[
        		\left\|u\right\|_s = \(\|S\|_{H^{s+2}(\RR^d)}^2 + \|A\|_{H^s(\RR^d)}^2\)^{1/2}.
    	\]

	 \begin{theorem}\label{thm:conv}
		 Let $s>d/2+1$, $\eps_{max}>0$, $u_0\in\Sigma_{s+2}$ and $0<T<T_{max}$ where
		 \[
            		T_{max} = \sup\{t>0 : \tau \mapsto \phi^0_\tau(u_0)\in L^\infty([0,t];\Sigma_{s+2})\}
        		\]
		and 
		$\phi_\tau^\eps$ denotes the flow at time $\tau$ of \eqref{eq:SClinear}.
		 There exists $C>0$  and $h_0>0$ such that the following error estimate holds true for any $\eps\in(0,\eps_{max}]$, any $h\in[0,h_0]$ and $n\in\NN$ satisfying $nh\leq T$:
		\[
			\begin{split}
				&\|(\varphi^1_h\circ\varphi^2_h\circ\varphi^3_h\circ \varphi^4_h)^n(u_0)-\phi^\eps_{nh}(u_0)\|_{s}\leq Ch.
			\end{split}
		\]
		The constants $C$ and $h_0$ do not depend on $\eps$.
	 \end{theorem}
	 \begin{remark}
		The constant $T_{max}$ appearing in Theorem \ref{thm:conv} is well-defined and positive since $\tau\mapsto \phi^0_\tau$ exists locally in time (see Theorem \ref{thm:existencecontrole}). 
	 \end{remark}
	 \begin{remark}
		The numerical analysis performed for the proof of Theorem \ref{thm:conv} can immediately be extended after the caustics for  $T_{max}\leq T$ and $\eps>0$ since the solution of \eqref{eq:SClinear} (as the one of \eqref{eq:GPElinear}) are global. Nevertheless, the constants $C$ and $h_0$ appearing in the result will not be independent on $\eps$ anymore. This point is illustrated in Section \ref{sec:num}.
	 \end{remark}
	  \begin{remark}
		The proof of Theorem \ref{thm:conv} can be adapted to any time-splitting method of order $1$ obtained after permutation of the four flows in \eqref{scheme1}. 
	 \end{remark}
	 Our proof is reminiscent of two previous results related to, on the one hand, splitting schemes for equations with Burgers nonlinearity \cite{2012_Holden_Lubich_Risebro} and on the other hand, splitting scheme for NLS in the semiclassical limit with \cite{MR3138106}.  Nonetheless, due to the finite-time existence of both exact and approximate flows, and to the peculiarity of the Lipschitz-type stability of the exact flows (see Lemma \ref{lem:1}), our proof follows a different path. In particular, we lean the approximate solutions on the exact one to  ensure that they do not  blow up. Besides, the application of Lady Windermere's fan argument is somehow hidden in an induction procedure. Finally, let us mention that, in spite of the fact that we do not specifically address this case, it is our belief that this result  
can be extended to Schrödinger equations with time dependent potentials, to the second-order scheme, to the Schrödinger equation with a nonlinearity of Hartree-type and to the weakly nonlinear Schrödinger equation (see also \cite[Remark  4.5]{MR3138106}).

The paper is organized as follows. In Section \ref{sec:proof}, we first give a theorem of well-posedness of the Cauchy problem for \eqref{eq:SClinear}. Then Theorem \ref{thm:conv} is proved using four technical lemmas. Section \ref{sect:proofs} is devoted to the proof of these lemmas. We illustrate the properties of our methods in Section \ref{sec:num}.
%
\section{Preparatiory results and proof of Theorem \ref{thm:conv}}
\label{sec:proof}
	\subsection{Notations}
	Assume that $\eps\in(0,\eps_{max}]$ and $s>d/2+1$. For the sake of simplicity, we keep the notation of all the flows independent of $\eps$. All the constants appearing in the proof depend on $\mathcal{V}$ but not on $\eps>0$.
	We denote 
	\[\begin{split}
		&\varphi^{ij}_h = \varphi^i_h\circ\varphi^j_h,\; \varphi^{ijk}_h = \varphi^{i}_h\circ\varphi^{jk}_h,\; 
		\\
		&\varphi^{1234}_h = \varphi^{1}_h\circ\varphi^{234}_h = \varphi^1_h\circ\varphi^2_h\circ\varphi^3_h\circ \varphi^4_h,
		\end{split}
	\]
	$\N_i$ is the possibly nonlinear operator related to $\varphi^i_h$ so that 
	\[
	\pa_h \varphi^i_h = \N_i \varphi^i_h.
	\]
	The quantities $\partial_h\varphi_h(u)$ and $\partial_2\varphi_h(u) $ are the Fréchet derivatives of $\varphi$ with respect  to $h$ and $u$.
	The commutator of the nonlinear operators $\N_i$ and $\N_j$ is given by
	\[
		[\N_i,\N_j](u) = D\N_i(u)\cdot \N_j(u) - D\N_j(u)\cdot \N_i(u).
	\]
	\subsection{Existence, uniqueness  and uniform boundedness results}
	The following theorem study some properties of the solutions of equations \eqref{eq:SClinear}.
	 \begin{theorem}\label{thm:existencecontrole}
    		Let $\eps_{max}>0$, $s>d/2+1$ and $u_0\in\Sigma_{s+2}$. The following two points are true.
    		\begin{enumerate}[(i)]
        		\item The quantity \be\label{eq:Tmax}
            	T_{max} = \sup\{t>0 :\phi^0(u_0)\in L^\infty([0,t];\Sigma_{s+2})\}
        		\ee
        		is well-defined and positive.
        		\item Let $0<T<T_{max}$. For all $\eps\in[0,\eps_{max}]$, there exists a unique  solution
       		 \[
           		\phi^\eps(u_0)\in C([0,T], \Sigma_{s+2})
        		\]
        	 	of the systems of equations \eqref{eq:SClinear}. Moreover, $\phi^\eps(u_0)$ is bounded in
        		\[
            		C([0,T],\Sigma_{s+2})
       		 \]
     		uniformly in $\eps\in[0,\eps_{max}]$.
    		\end{enumerate}
    	\end{theorem}
	 The proof of Theorem \ref{thm:existencecontrole} is given in Section \ref{sec:thm:existencecontrole}.
	\subsection{The main lemmas}
	In this subsection, we present the main ingredients needed in the proof of Theorem \ref{thm:conv}. Their proof is postponed to Section \ref{sect:proofs}. 
	\begin{lemma}\label{lem:0}
		Let $M>0$ and $s>d/2+1$. There exist $h_1 = h_1(M)>0$ such that for any $\eps\in(0,\eps_{max}]$ and any $u_0\in \Sigma_{s}$ satisfying 
		\[
			\|u_0\|_{{s}}\leq M,
		\]
		we have that the solution $\phi_t(u_0)$ of equation \eqref{eq:SClinear} is well-defined on $[0,h_1]$ and for all $t\in[0,h_1]$
		\[
			\|\phi_t(u_0)\|_{{s}}\leq 2M.
		\]
	\end{lemma}
	\begin{lemma}\label{lem:1}
		Let $M>0$ and $s>d/2+1$. There exist $C_2 = C_2(M)>0$ such that for any $\eps\in(0,\eps_{max}]$, any solutions $\phi_t(u_1)\in L^\infty([0,T],\Sigma_{s+1})$ and $\phi_t(u_2)\in L^\infty([0,T],\Sigma_{s})$ of equation \eqref{eq:SClinear}, satisfying for all $t\in[0,T]$
		\[
			\| \phi_t(u_1)\|_{{s+1}}+ \|\phi_t(u_2)\|_{{s}}\leq M
		\]
		we have
		\[
			\|\phi_t(u_1)-\phi_t(u_2)\|_{{s}}\leq \| u_1-u_2\|_{{s}} \exp(C_2t).
		\]
	\end{lemma}
	\begin{remark}
		Let us insist on the fact that in Lemma \ref{lem:1}, we have to control $\phi_t(u_1)$ in $\Sigma _{s+1}$ and $\phi_t(u_2)$ in $\Sigma _{s}$ to get Lipschitz-type stability in $\Sigma_{s}$.
	\end{remark}
	\begin{lemma}\label{lem:2}
		Let $M>0$ and $s>d/2+1$. There exist $h_{3} = h_{3}(M)>0$ and $C_{3} = C_{3}(M)>0$ such that for any $\eps\in(0,\eps_{max}]$, any $u_0\in \Sigma_{s}$ satisfying $\|u_0\|_{s}\leq M$ and any $0\leq t\leq h_{3}$, we have
		\begin{enumerate}[(a)]
			\item $\|\varphi^{1234}_{t}(u_0)\|_{s}\leq 8M$.
			\item Furthermore, if $u_0\in\Sigma_{s+2}$, then 
			\[
				\|\varphi^{1234}_{t}(u_0)\|_{s+2}\leq  \exp\(C_3t\)\(\|u_0\|_{s+2}+ t\|\mathcal{V}\|_{H^{s+4}}\).
			\]
		\end{enumerate}
	\end{lemma}
	\begin{lemma}\label{lem:3}
		Let $M>0$ and $s>d/2+1$. There exist $h_4 = h_4(M)>0$ and $K_4 = K_4(M)>0$ such that for any $\eps\in(0,\eps_{max}]$ and any $u_0\in\Sigma_{s+2}$ satisfying 
		\[
			\|u_0\|_{s+2}\leq M,
		\]
		we have for any $t\in[0,h_4]$ that
		\[
			\|\phi_t(u_0)-\varphi_{t}^{1234}(u_0)\|_{s}\leq K_4t^2.
		\]
	\end{lemma}
	\subsection{Proof of Theorem \ref{thm:conv}}
	Let us denote    
    	\be\label{bound}
            M_{s}^{\eps}(T) := \sup\{\|\phi_t^\eps(u_0)\|_{s}:\; 0\leq t\leq T\}.
    	\ee
    	for $\eps\geq0$ and $T\geq0$.
    %
    %
    %
  
    %
    %
  %
	%
	%
	%
	Let $s>d/2+1$, $\eps\in(0,\eps_{max}]$, $u_0\in \Sigma_{s+2}$, $n\in \NN$ and $h>0$ be such that $nh\leq T<T_{max}$ (see \eqref{eq:Tmax}). 
	By Theorem \ref{thm:existencecontrole}, there exist $M_{s}$, $M_{s+1}$ and $M_{s+2}$ independent of $\eps\in(0,\eps_{max}]$ such that for all $\eps\in(0,\eps_{max}]$, 
	\[
		M^\eps_s\leq M_s,\;M^\eps_{s+1}\leq M_{s+1} \mbox{ and }M^\eps_{s+2}\leq M_{s+2}
	\]
	(see \eqref{bound}).
	We denote
	\[\begin{split}
		& C = C_3(2M_s),\\
		&c_0 =  \|u_0\|_{s+2}\exp\(CT\) + \|\mathcal{V}\|_{H^{s+4}}e^{2TC}/C,\\
		&C' = C_2(M_{s+1} + 4M_s),\\
		&\widetilde c = K_4(c_0)a e^{C'T}/C'.
	\end{split}\]
	Assume that 
	\be\label{hyp:h}
		0\leq h \leq \min \(h_3(c_0),M_s/\widetilde c,h_1(2M_s),h_4(c_0)\).
	\ee
	Here, $h_1$, $C_2$, $h_3$, $h_4$ and $K_4$  are defined in Lemmas \ref{lem:0}, \ref{lem:1}, \ref{lem:2} and  \ref{lem:3}.

	We show by induction on $0\leq k\leq n$ that
	\begin{enumerate}[(i)]
		\item \label{induction2} $(\varphi^{1234}_{h})^k(u_0)$ is well-defined, belongs to $\Sigma_{s+2}$ and
			\[
				\|(\varphi^{1234}_{h})^{k}(u_0)\|_{s+2}\leq \|u_0\|_{s+2}\exp\(Ckh\) + h \|\mathcal{V}\|_{H^{s+4}}\frac{e^{(k+1)hC} - e^{hC}}{e^{hC}-1}\leq c_0,
			\]
		\item \label{induction1} $\|\phi_{kh}(u_0) - (\varphi^{1234}_{h})^k(u_0)\|_s\leq h^2K_4(c_0)\frac{e^{C'hk}-1}{e^{C'h}-1}\leq \widetilde c h$,
	\end{enumerate}
	and Theorem \ref{thm:conv} follows then from point \eqref{induction1} with $k = n$.
	%
	%
		 
		The induction hypothesis is true for $k = 0$. Let us assume points \eqref{induction2} and \eqref{induction1}  true for $0\leq k\leq n-1$.

		Lemma \ref{lem:2}, point \eqref{induction2} and \eqref{hyp:h} ensure that 
		\[
			 (\varphi^{1234}_{h})^{k+1}(u_0)
		\]
		 is well-defined and belongs to $\Sigma_{s+2}$.
		By Point \eqref{induction1} and \eqref{hyp:h}, we have
		\[
			\|(\varphi^{1234}_{h})^k(u_0)\|_s\leq M_s + \|\phi_{kh}(u_0) - (\varphi^{1234}_{h})^k(u_0)\|_s\leq 2M_s.
		\]
		By Lemma \ref{lem:2} and \eqref{hyp:h}, we have 
		\[
			\|(\varphi^{1234}_{h})^{k+1}\|_{s+2}\leq  \exp\(Ch\)\(\|(\varphi^{1234}_{h})^k(u_0)\|_{s+2}+ h\|\mathcal{V}\|_{H^{s+4}}\)
		\]
		and point \eqref{induction2} ensures that 
		\[\begin{split}
			\|(\varphi^{1234}_{h})^{k+1}\|_{s+2}\leq  \|u_0\|_{s+2}\exp\(C(k+1)h\) + h \|\mathcal{V}\|_{H^{s+4}}\frac{e^{(k+2)hC} - e^{hC}}{e^{hC}-1}\leq c_0.
		\end{split}\]
		By Lemma \ref{lem:0} and \eqref{hyp:h}, $h'\mapsto\phi_{h'}\circ(\varphi^{1234}_{h})^k(u_0)$ is well-defined and satisfies for all $0\leq h'\leq h$
		\[
			\|\phi_{h'}\circ(\varphi^{1234}_{h})^k(u_0)\|_{{s}}\leq 4M_s.
		\]
		By Lemma \ref{lem:1}, we obtain that
		\[\begin{split}
			&\|\phi_{h(k+1)}(u_0)-\phi_h\circ(\varphi^{1234}_{h})^k(u_0)\|_{{s}}\leq \|\phi_{hk}(u_0)-(\varphi^{1234}_{h})^k(u_0)\|_{s} \exp(C'h).
		\end{split}\]
		By Lemma \ref{lem:3}, point \eqref{induction2} and \eqref{hyp:h}, we get 
		\[
			\|\phi_{h}\circ(\varphi^{1234}_{h})^{k}(u_0) - \varphi^{1234}_{h}\circ(\varphi^{1234}_{h})^{k}(u_0)\|_{s}\leq K_4(c_0)h^2,
		\]	
		so that
		\[\begin{split}
			&\|\phi_{h(k+1)}(u_0)-(\varphi^{1234}_{h})^{k+1}(u_0)\|_{{s}}\leq K_4(c_0)h^2 + \|\phi_{hk}(u_0)-(\varphi^{1234}_{h})^k(u_0)\|_{s} \exp(C'h).
		\end{split}\]
		By point \eqref{induction1}, we have then that
		\[
			\|\phi_{h(k+1)}(u_0)-(\varphi^{1234}_{h})^{k+1}(u_0)\|_{{s}}\leq K_4(c_0)h^2\(\frac{e^{C'h(k+1)}-1}{e^{C'h}-1}\).
		\]
		Thus, points \eqref{induction2} and \eqref{induction1}  are true for $k+1$.
		%
		%

	\section{Proof of the main lemmas} \label{sect:proofs}
	\subsection{Auxiliary results}
	Let us denote by 
	$
		\<\cdot ,\cdot\>
	$
	the $L^2$ scalar product, for $s>0$
	\[
		\Lambda^s = (1-\Delta)^{s/2},
	\]	
	\be\label{eq:proj}
			\Pi_1u = S,\; \Pi_2u = A, \mbox{ for }u = \(\begin{array}{c}S\\A\end{array}\)
	\ee
	and
	\be\begin{split}\label{eq:produitscalaire}
			\<u_1,u_2\>_s &= \<\Pi_1u_1 , \Pi_1u_2\> 
			\\
			&\quad+  \<\Lambda^{s+1}\nabla\Pi_1u_1 , \Lambda^{s+1}\nabla\Pi_1u_2\> +  \RE\<\Lambda^{s}\Pi_2u_1 , \Lambda^{s}\Pi_2u_2\>.
	\end{split}\ee

	We recall  two points that will be of constant use in the following: the Sobolev space $H^{s}\subset L^\infty$ is an algebra for $s>d/2$
	and  the Kato-Ponce \cite{kato1988commutator} inequality holds true:
	\begin{proposition}\label{prop:KatoPonce}
		Let ${s_0}> d/2+1$. There is $c>0$ such that for all $f\in H^{{s_0}}(\RR^d)$ and $g\in H^{{s_0}-1}(\RR^d)$
		\[
			\|\Lambda^{{s_0}}(fg)-f\Lambda^{{s_0}}g\|_{L^2}\leq c(\|\nabla f\|_{L^\infty}\|g\|_{H^{{s_{0}}-1}}+\|f\|_{H^{{s_{0}}}}\|g\|_{L^\infty}).
		\]
	\end{proposition}
	The following lemmas will be used several times in our proof.
	\begin{lemma}\label{eq:burger}
		Let $s_0>d/2+1$. There is $C>0$ such that for all $v_0$, $v_1$ and $R\in L^\infty([0,h_0],H^{s_0}(\RR^d)^d)$ satisfying
		\[
			\pa_t v_0 + (v_1\cdot\nabla)v_0 = R,
		\]
		we have
		\[\begin{split}
			\pa_t\|v_0\|_{H^{s_0}}^2
			&\leq C\(\|v_0\|_{H^{s_0}}^2\|\nabla v_1\|_{L^\infty}+ \|v_0\|_{H^{s_0}}\|v_1\|_{H^{s_0}}\|\nabla v_0\|_{L^\infty}\)+ 2\<\Lambda^{s_0}v_0,\Lambda^{s_0}R\>
			\\
			&\leq C\|v_0\|_{H^{s_0}}^2\|v_1\|_{H^{s_0}}+ 2\<\Lambda^{s_0}v_0,\Lambda^{s_0}R\>.
		\end{split}\]
	\end{lemma}
	\begin{proof}
		We have by integration by parts that
		\[\begin{split}
			\pa_t\frac{\|v_0\|_{H^{s_0}}^2}{2}&= \<\Lambda^{s_0}v_0,\Lambda^{s_0}\pa_t v_0\> = -\<\Lambda^{s_0}v_0,\Lambda^{s_0}(v_1\cdot\nabla)v_0\> + \<\Lambda^{s_0}v_0,\Lambda^{s_0}R\>\\
			& \leq \frac{1}{2}\int_{\RR^d}|\Lambda^{s_0}v_0|^2\DIV v_1 + \|v_0\|_{H^{s_0}}\|[\Lambda^{s_0}, (v_1\cdot\nabla)]v_0\|_{L^2}+ \<\Lambda^{s_0}v_0,\Lambda^{s_0}R\>.
		\end{split}\]
		Proposition \ref{prop:KatoPonce} ensures that
		\[\begin{split}
			\pa_t\frac{\|v_0\|_{H^{s_0}}^2}{2}
			&\leq c\(\|v_0\|_{H^{s_0}}^2\|\nabla v_1\|_{L^\infty}+ \|v_0\|_{H^{s_0}}\|v_1\|_{H^{s_0}}\|\nabla v_0\|_{L^\infty}\)+ \<\Lambda^{s_0}v_0,\Lambda^{s_0}R\>.
		\end{split}\]
	\end{proof}
	\begin{lemma}\label{eq:transport}
		Let $s_0>d/2+1$. There exists $C>0$ such that for all 
		$A\in  W^{1,\infty}([0,h_0],H^{s_0}(\RR^d))$, 
		$v_1\in L^\infty([0,h_0],H^{s_0+1}(\RR^d)^d)$ and $R\in L^\infty([0,h_0],H^{s_0}(\RR^d))$ satisfying 
		\[
			\pa_t A + v_1\cdot \nabla A + A\frac{\DIV v_1}{2} = R,
		\]
		we have,
		\[\begin{split}
			\pa_t\|A\|_{H^{s_0}}^2
			&\leq C\(\|A\|_{H^{s_0}}^2\| v_1\|_{W^{2,\infty}}+\|A\|_{H^{s_0}}\|v_1\|_{H^{s_0+1}}\|A\|_{W^{1,\infty}}\)
			+ 2\RE \<\Lambda^{s_0}A,\Lambda^{s_0}R\>\\
			&\leq C\|A\|_{H^{s_0}}^2\|v_1\|_{H^{s_0+1}}+ 2\RE \<\Lambda^{s_0}A,\Lambda^{s_0}R\>.
		\end{split}\]

	\end{lemma}
	\begin{proof}
		We have by integration by parts that
		\[\begin{split}
			\pa_t\frac{\|A\|_{H^{s_0}}^2}{2}&= \RE \<\Lambda^{s_0}A,\Lambda^{s_0}\pa_t A\> = -\RE \<\Lambda^{s_0}A,\(v_1\cdot \nabla  + \frac{\DIV v_1}{2}\)\Lambda^{s_0}A\>\\
			& \quad-\RE \<\Lambda^{s_0}A,\left[\Lambda^{s_0},\(v_1\cdot \nabla  + \frac{\DIV v_1}{2}\)\right]A\> + \RE \<\Lambda^{s_0}A,\Lambda^{s_0}R\>\\
			& \leq\|A\|_{H^{s_0}}\left\|\left[\Lambda^{s_0},\(v_1\cdot \nabla  + \frac{\DIV v_1}{2}\)\right]A\right\|_{L^2} + \RE \<\Lambda^{s_0}A,\Lambda^{s_0}R\>.
		\end{split}\]
		Proposition \ref{prop:KatoPonce} ensures that
		\[\begin{split}
			\pa_t\frac{\|A\|_{H^{s_0}}^2}{2}
			&\leq C\|A\|_{H^{s_0}}(\|\nabla v_1\|_{L^\infty}\|\nabla A\|_{H^{s_0-1}}+\|v_1\|_{H^{s_0}}\|\nabla A\|_{L^\infty})\\
			&\qquad+C\|A\|_{H^{s_0}}(\|\nabla (\DIV v_1)\|_{L^\infty}\|A\|_{H^{s_0-1}}+\|\DIV v_1\|_{H^{s_0}}\|A\|_{L^\infty})\\
			&\qquad+ \RE \<\Lambda^{s_0}A,\Lambda^{s_0}R\>\\
			&\leq C\(\|A\|_{H^{s_0}}^2\| v_1\|_{W^{2,\infty}}+\|A\|_{H^{s_0}}\|v_1\|_{H^{s_0+1}}\|A\|_{W^{1,\infty}}\)+ \RE \<\Lambda^{s_0}A,\Lambda^{s_0}R\>\\
			&\leq C\|A\|_{H^{s_0}}^2\|v_1\|_{H^{s_0+1}}+ \RE \<\Lambda^{s_0}A,\Lambda^{s_0}R\>.
		\end{split}\]
	\end{proof}
		%
		%
	\subsection{Study of the equation \eqref{eq:SClinear}}
	Let us prove Lemma \ref{lem:0}.
	%
	%
	\begin{proof}
		By the Cole-Hopf transform, we get that $w^\eps = \exp\(-\frac{\se}{2\eps^2}\)-1$ is the solution of
		\[
			\pa_t w^\eps = \eps^2\Delta w^\eps + \frac{\mathcal{V}}{2\eps^2}\(w^\eps + 1\),\quad w^\eps(0) =  \exp\(-\frac{S_0}{2\eps^2}\)-1,
		\]
		Hence, global existence and uniqueness of the solution $\se$ of \eqref{eq:SClinearphase} for fixed $\eps\in(0,\eps_{max}]$, follows from standard semi-group theory. The function $v^\eps = \nabla \se$ solves
		\[
			\pa_t v^\eps + (v^\eps\cdot \nabla) v^\eps + \nabla \mathcal{V} = \eps^2\Delta v^\eps.
		\]
		 Since $s>d/2$, Lemma \ref{eq:burger} and an integration by parts ensure that
		\[\begin{split}
			\pa_t\|v^\eps\|_{H^{s+1}}^2&\leq c\|v^\eps\|_{H^{s+1}}^3+ \<\Lambda^{s+1}v^\eps,\Lambda^{s+1} \(-\nabla \mathcal{V} + \eps^2\Delta v^\eps\)\>\\
			&\leq c\|v^\eps\|_{H^{s+1}}^3+ \|v^\eps\|_{H^{s+1}}\|\mathcal{V}\|_{H^{s+2}}.
		\end{split}\]
		By \eqref{eq:SClinearphase}, we also have that
		\[
			\pa_t \frac{\|\se\|_{L^2}^2}{2}\leq \|\se\|_{L^2}\(\|\mathcal{V}\|_{L^2} + \|v^\eps\|_{L^4}^2/2\)
		\]
		so that
		\[
			\pa_t \|\se\|_{H^{s+2}}^2\leq c\|\mathcal{V}\|_{H^{s+2}}\|\se\|_{H^{s+2}} + c\|\se\|_{H^{s+2}}^3.
		\]
		The global existence and the uniqueness of a solution $\ae$ of equation \eqref{eq:SClinearamplitude} follows from the fact that 
		\[
			\pe = \ae \exp\(i\se/\eps\)
		\] 
		satisfies equation \eqref{eq:GPElinear}.
		By Lemma \ref{eq:transport}, recalling that $s>d/2+1$, we also have
		\[\begin{split}
			\pa_t\|\ae\|^2_{H^{s}} \leq c\|\ae\|_{H^{s}}^2\|\se\|_{H^{s+2}}+ \RE \<\Lambda^{s}\ae,\Lambda^{s}R\>.
		\end{split}\]
		where $R = \frac{i\eps\Delta \ae}{2} -i\eps\ae\Delta\se$ so that an integration by parts gives us
		\[\begin{split}
			\pa_t\|\ae\|^2_{H^{s}} 
			&\leq c\|\ae\|_{H^{s}}^2\|\se\|_{H^{s+2}}.
		\end{split}\]
		 We obtain that
		 \[
		 	\pa_t\|\phi_{t}(u_0)\|^2_{s}\leq  c_1\|\phi_{t}(u_0)\|_{s}\|\mathcal{V}\|_{H^{s+2}} + c_2\|\phi_{t}(u_0)\|^3_{s}
		 \]
		 and
		\[
		 	\pa_t\|\phi_{t}(u_0)\|_{s}\leq  c_1\|\mathcal{V}\|_{H^{s+2}}+ c_2\|\phi_{t}(u_0)\|^2_{s}.
		 \]
		We get then that
		\[
			\|\phi_{t}(u_0)\|_{s}\leq \sqrt{\frac{c_1\|\mathcal{V}\|_{H^{s+2}}}{c_2}}\tan \(t\sqrt{c_1c_2\|\mathcal{V}\|_{H^{s+2}}} + \arctan \(M\sqrt{\frac{c_2}{c_1\|\mathcal{V}\|_{H^{s+2}}}}\)\)
		\]
		so that there is $h_1 = h_1(M)>0$ such that for all $0\leq t \leq h_1$
		\[
			\|\phi_{t}(u_0)\|_{s}\leq 2M.	
		\]
	\end{proof}
	The following result will be used several times and in particular for the proof of the stability of equation \eqref{eq:SClinear} in Lemma \ref{lem:1}. 		
	\begin{lemma}\label{lem:stabaux}
		Let $s_0>d/2+1$. Let $u_1 = (S_{1},A_1)$ be in $L^\infty([0,T],\Sigma_{s_0+1})$, $u_2 = (S_{2},A_2)$, $(R_{1,S},R_{1,A})$ and $(R_{2,S},R_{2,A})$ be in $ L^\infty([0,T],\Sigma_{s_0})$. Assume moreover that for $i = 1,2$
		\[\begin{split}
			\pa_t S_i + \frac{|\nabla S_i|^2}{2} = R_{i,S},\\
			\pa_t A_i  + \nabla S_i\cdot \nabla A_i + A_i \frac{\Delta S_i}{2} = R_{i,A}.
		\end{split}\]
		Then, we have
		\[\begin{split}
			&\pa_t\|u_1-u_2\|_{s_0}^2\leq c\|u_1-u_2\|_{s_0}^2\(\|u_1\|_{s_0+1}+ \|u_2\|_{s_0}\) + 2\<u_1-u_2,R_{1}-R_{2}\>_{s_0}
			%
			%
			%
			%
			 %
		\end{split}\]
		where $R_i = (R_{i,S},R_{i,A})^T$.
	\end{lemma}
	\begin{proof}
		Let $s_0>d/2+1$. 
		Let us define $v_1 = \nabla S_1$, $v_2 = \nabla S_2$, $w = v_1-v_2$, $B = A_1-A_2$ and $u = u_1-u_2$.
		
		We have that
		\[\begin{split}
			\pa_t w  &= - (v_1\cdot \nabla)v_1 +  (v_2\cdot \nabla)v_2 +\nabla\(R_{1,S}-R_{2,S}\)\\
			& = - (v_2\cdot \nabla)w -  (w\cdot \nabla)v_1 +\nabla\(R_{1,S}-R_{2,S}\)
		\end{split}\]
		 and Lemma \ref{eq:burger} ensures that
		 \[\begin{split}
		 	\pa_t\|w\|_{H^{s_0+1}}^2&\leq c\|w\|_{H^{s_0+1}}^2\| v_2\|_{H^{s_0+1}}+ 2\<\Lambda^{s_0+1}w,\Lambda^{s_0+1}R\>.
		 \end{split}\]
		 where $R = -  (w\cdot \nabla)v_1 + \nabla\(R_{1,S}-R_{2,S}\).$  
		 We also have that
		 \[\begin{split}
			\|(w\cdot \nabla)v_1\|_{H^{s_0+1}}
			&\leq c\|w\|_{H^{s_0+1}}\|v_1\|_{H^{s_0+2}}  
		 \end{split}\]
		 and
		\[\begin{split}
		 	\pa_t\|w\|_{H^{s_0+1}}^2&\leq c\|w\|_{H^{s_0+1}}^2\(\| S_1\|_{H^{s_0+3}}+ \|S_2\|_{H^{s_0+2}}\) +  2\<\Lambda^{s_0+1}w,\Lambda^{s_0+1}\nabla \(R_{1,S}-R_{2,S}\)\>.
		 \end{split}\]
		 We also have
		 \[
		 	\pa_t (S_1-S_2) = -\frac{1}{2}(v_1+v_2)\cdot w+ \(R_{1,S}-R_{2,S}\)
		 \]
		 so that
		 \[
		 	\pa_t \|S_1-S_2\|_{L^2}^2\leq c\|S_1-S_2\|_{L^2}\|w\|_{L^{2}} \(\| S_1\|_{W^{1,\infty}}+ \|S_2\|_{W^{1,\infty}}\) + 2\<S_1-S_2,R_{1,S}-R_{2,S}\>
		 \]
		 and then
		 \[\begin{split}
		 	&\pa_t \|S_1-S_2\|_{H^{s_0+2}}^2
			\leq C\|S_1-S_2\|_{H^{s_0+2}}^2 \(\| S_1\|_{H^{s_0+3}}+ \|S_2\|_{H^{s_0+2}}\)
			\\
			&\quad+2\<S_1-S_2,R_{1,S}-R_{2,S}\> +  2\<\Lambda^{s_0+1}\nabla \(S_1-S_2\),\Lambda^{s_0+1}\nabla \(R_{1,S}-R_{2,S}\)\>
		 \end{split}\]
		 Let us study $B$, we have
		\[\begin{split}
			\pa_t B + \nabla S_2\cdot \nabla B + \frac{\Delta S_2}{2}B = R
		\end{split}\]
		where
		\[
			R = -w\cdot\nabla A_1 - \frac{\DIV (w)}{2}A_1 + \(R_{1,A}-R_{2,A}\)
		\]
		Hence, we obtain by Lemma \ref{eq:transport} 
		\[\begin{split}
			\pa_t\|B\|_{H^{s_0}}^2&\leq c\|B\|_{H^{s_0}}^2\|S_2\|_{H^{s_0+2}}+ 2\RE \<\Lambda^{s_0}B,\Lambda^{s_0}R\>\\
			&\leq c\|B\|_{H^{s_0}}^2\|S_2\|_{H^{s_0+2}}+  c\|B\|_{H^{s_0}}\|w\|_{H^{s_0+1}}\|A_1\|_{H^{s_0+1}} \\
			&\qquad + 2\RE \<\Lambda^{s_0}B,\Lambda^{s_0}\(R_{1,A}-R_{2,A}\)\>
		\end{split}\]
		and
		\[\begin{split}
			&\pa_t\|u_1-u_2\|_{s_0}^2\leq c\|u_1-u_2\|_{s_0}^2\(\|u_1\|_{s_0+1}+ \|u_2\|_{s_0}\)+2\<u_1-u_2,R_1-R_2\>_{s_0}
			%
			%
			%
			%
			 %
		\end{split}\]
		The result follows.
	\end{proof}
	Let us study now the stability of equation \eqref{eq:SClinear} and prove Lemma \ref{lem:1}.
	%
	%
	%
	%
	\begin{proof}
		Let $s>d/2+1$  and $\eps\in(0,\eps_{max}]$.
		Let us define for $i = 1,2$
		\[\begin{split}
			R_{i,S} &= -\mathcal{V} + \eps^2\Delta S_{i},
			\\
			R_{i,A} &=i\eps\frac{\Delta A_i}{2} -i\eps A_{i}\Delta S_i.
		\end{split}\]
		We apply Lemma \ref{lem:stabaux} with $s_0 = s$. 
		We have by integrations by parts that
		\[\begin{split}
			& \<S_1-S_2,R_{1,S}-R_{2,S}\>+  \<\Lambda^{s+1}\nabla \(S_1-S_2\),\Lambda^{s+1}\nabla \(R_{1,S}-R_{2,S}\)\> \leq 0,\\
		\end{split}\]
		and
		\[\begin{split}
			& \RE \<\Lambda^{s}\(A_1-A_2\),\Lambda^{s}\(R_{1,A}-R_{2,A}\)\>
			\leq c\|A_1-A_2\|^2_{H^{s}}\|S_1\|_{H^{s+2}}\\
			&\qquad\qquad+c\|A_1-A_2\|_{H^{s}}\|S_1-S_2\|_{H^{s+2}}\|A_2\|_{H^{s}}.
		\end{split}\]
		so that
		\[\begin{split}
			&\pa_t\|\phi_t(u_1)-\phi_t(u_2)\|_{s}^2\leq c\|u_1-u_2\|_{s}^2\(\|\phi_t(u_1)\|_{s+1}+ \|\phi_t(u_2)\|_{s}\)
		\end{split}\]
		and the result follows.
	\end{proof}
	%
	%
	%
	\subsection{Study of the numerical flow $\varphi^{1234}$}
	The following lemma is inspired by the work of Holden, Lubich and Risebro \cite{2012_Holden_Lubich_Risebro}.
	\begin{lemma}\label{lem:flot1}
		Let $s_0>d/2+1$ and $M>0$. There exists $h_{5} = h_{5}(M)>0$ such that for any $u_0\in \Sigma_{s_0}$ satisfying $\|u_0\|_{s_0}\leq M$ and any $0\leq t\leq h_{5}$, the following two points are true.
		\begin{enumerate}[(i)]
			\item We have that $\|\varphi^1_{t}(u_0)\|_{s_0}\leq 2M$.
			\item Let $s_1\geq s_0$. There is $C_5 = C_5(M)>0$ such that if $u_0\in\Sigma_{s_1}$, then 
			\[\|\varphi^1_{t}(u_0)\|_{s_1}\leq \exp\(C_{5}t\)\|u_0\|_{s_1}.\]
		\end{enumerate}
	\end{lemma}
	\begin{proof}
		The existence of the solution $S$ of \eqref{eq:SCflot1phase} follows for instance from the method of characteristics. Lemma \ref{eq:burger} ensures that for $s>d/2+1$
		\[
			\pa_t\|\nabla S\|_{H^{s+1}}^2\leq c\|\nabla S\|_{H^{s+1}}^2\|\nabla (\nabla S)\|_{L^\infty}\leq C\|\nabla S\|_{H^{s+1}}^2\|S(t)\|_{W^{2,\infty}}.
		\]
		We also have
		\[
			\pa_t \|S\|_{L^2}^2\leq c\|S(t)\|_{L^2}\|\nabla S(t)\|_{L^4}^2
		\]
		so that
		\[
			\pa_t \|S\|_{H^{s+2}}^2\leq C\|S(t)\|_{H^{s+2}}^2\|S(t)\|_{W^{2,\infty}}.
		\]
		The remaining of the proof follows exactly the same lines as the one of Lemma \ref{lem:0}. By Lemma \ref{eq:transport} and an integration by parts, we have
		\[\begin{split}
			\pa_t\|A\|^2_{H^{s}} 
			&\leq C\(\|A\|_{H^{s}}^2\| S\|_{W^{3,\infty}}+\|A\|_{H^{s}}\|S\|_{H^{s+2}}\|A\|_{W^{1,\infty}}\)\\
			&\leq C\|\varphi^1_{t}(u_0)\|_{{s}}^2\|\varphi^1_{t}(u_0)\|_{W^{3,\infty}\times W^{1,\infty}}\\
		\end{split}\]
		and
		\[\begin{split}
			\pa_t\|\varphi^1_{t}(u_0)\|_{{s}}^2
			&\leq C\|\varphi^1_{t}(u_0)\|_{{s}}^2\|\varphi^1_{t}(u_0)\|_{W^{3,\infty}\times W^{1,\infty}}\\
			&\leq C\|\varphi^1_{t}(u_0)\|_{{s}}^3.
		\end{split}\]
		Taking $s = s_0$, we get that
		\[
			\|\varphi^1_{t}(u_0)\|_{{s_0}}
			\leq \frac{M}{1-cMt}
		\]
		and there is $h_5= h_5(M)>0$ such that for all $t\in[0,h_9]$
		\[
			\|\varphi^1_{t}(u_0)\|_{{s_0}}\leq 2M.
		\]
		We also obtain for $s = s_1\geq s_0>d/2+1$ and $t\in [0,h_9]$ that
		\[\begin{split}
			\pa_t\|\varphi^1_{t}(u_0)\|_{{s_1}}^2
			&\leq C\|\varphi^1_{t}(u_0)\|_{{s_1}}^2\|\varphi^1_{t}(u_0)\|_{{s_0}}
			\\
			&\leq 2CM\|\varphi^1_{t}(u_0)\|_{{s_1}}^2.
		\end{split}\]
		and the result follows from Gronwall's Lemma.
	\end{proof}
	We immediately get the following result for the second and the third flows.
	\begin{lemma}\label{lem:flot23}
		Let $s_0>0$ and $M>0$. There is $h_{6} = h_{6}(M)$ such that for any $u_0\in \Sigma_{s_0}$ satisfying $\|u_0\|_{s_0}\leq M$ any $0\leq t\leq h_{6}$, the following two points holds true.  
		\begin{enumerate}[(i)]
			\item $\|\varphi^2_{t}(u_0)\|_{s_0}\leq M$ and $\|\varphi^3_{t}(u_0)\|_{s_0}\leq 2M$,
			\item Let $s_1\geq 0$. If moreover $u_0\in\Sigma_{s_1}$, then, we have 
			\[\|\varphi^2_{t}(u_0)\|_{s_1}\leq \|u_0\|_{s_1}\mbox{ and }\|\varphi^3_{t}(u_0)\|_{s_1}\leq \|u_0\|_{s_1} + t\|\mathcal{V}\|_{H^{s_1+2}}.\]
		\end{enumerate}
	\end{lemma}

	The following lemma study the fourth flow.
	\begin{lemma}\label{lem:flot4}
		Let $s_0>d/2+1$ and $M>0$. There exists $h_{7} = h_{7}(M)>0$ such that for any $u_0\in \Sigma_{s_0}$ satisfying $\|u_0\|_{s_0}\leq M$ and any $0\leq t\leq h_{7}$, the following two points holds true. 
		\begin{enumerate}[(i)]
			\item $\|\varphi^4_{t}(u_0)\|_{s_0}\leq 2M$,
			\item Let $s_1\geq s_0$. There is $C_{7} = C_{7}(M)>0$ such that if $u_0\in \Sigma_{s_1}$,
			\[\|\varphi^4_{t}(u_0)\|_{s_1}\leq \exp\(C_{7}t\)\|u_0\|_{s_1}.\]
		\end{enumerate}
	\end{lemma}
	\begin{proof}
		Let $s>d/2+1$.
		By integration by parts, we have $\pa_t\|S(h)\|_{H^{s+2}}^2\leq 0$ and
		\[\begin{split}
			&\pa_t\frac{\|A\|^2_{H^{s}}}{2} =  \RE\<\Lambda^{s}A,\Lambda^{s}\(-i\eps A\Delta S\)\> = \RE\<\Lambda^{s}A,[\Lambda^{s},-i\eps \Delta S]A\>\\
			&\leq c\|A\|_{H^{s}}\(\|\nabla (\Delta S)\|_{L^\infty}\|A\|_{H^{s-1}}+\|\Delta S\|_{H^{s}}\|A\|_{L^\infty}\)\\
			&\leq c\|A\|_{H^{s}}^2\|S\|_{W^{3,\infty}}+c\|A\|_{H^{s}}\|S\|_{H^{s+2}}\|A\|_{L^\infty}.
		\end{split}\]
		We obtain for $s = s_0$ that
		\[
			\pa_t \|\varphi^4_t(u_0)\|_{s_0}^2\leq c \|\varphi^4_t(u_0)\|_{s_0}^3,
		\]
		for $s = s_1$ that
		\[
			\pa_t \|\varphi^4_t(u_0)\|_{s_1}^2\leq c \|\varphi^4_t(u_0)\|_{s_1}^2\|\varphi^4_t(u_0)\|_{s_0}
		\]
		and the result follows from the arguments of the end of the proof of Lemma \ref{lem:flot1}.
	\end{proof}
	Taking $s_0 = s$ and $s_1 = s_0+2$, we immediately get Lemma \ref{lem:2}  combining Lemmas \ref{lem:flot1}, \ref{lem:flot23} and \ref{lem:flot4}.
	\subsection{Proof of Theorem \ref{thm:existencecontrole}}\label{sec:thm:existencecontrole}
	%
	%
	%
	%
    %
    %
    %
    %
    %
		Let $M>0$. Lemma \ref{lem:0} ensures that there is $h_1 = h_1(M)>0$ such that for any $\eps\in(0,\eps_{max}]$ and any $u_0\in\Sigma_{s+2}$ satisfying $\|u_0\|_{s+2}\leq M$, the solutions $t\mapsto\phi_t^\eps(u_0)$ of equation \eqref{eq:SClinear} are well-defined in $L^\infty([0,h_1],\Sigma_{s+2})$ and uniformly bounded with respect to $\eps$.

		Let $\eps,\eps'\in(0,\eps_{max}]$, $u_0,u_0'\in \Sigma_{s+2}$ such that $\|u_0\|_{s+2}\leq M$ and  $\|u_0'\|_{s+2}\leq M$. We define $(S^\eps,A^\eps)^T = \phi^\eps(u_0)$, $(S^{\eps'},A^{\eps'})^T = \phi^{\eps'}(u_0')$ and
		\[\begin{split}
			&R_{1,S} = -\mathcal{V} + \eps^2\Delta S^\eps, \quad R_{2,S} = -\mathcal{V} + \eps'^2\Delta S^{\eps'},\\
			&R_{1,A} =i\eps\frac{\Delta A^\eps}{2} -i\eps A^\eps\Delta S^\eps, \quad R_{2,A} =i\eps'\frac{\Delta A^{\eps'}}{2} -i\eps' A^{\eps'}\Delta S^{\eps'}.
		\end{split}\]
		We apply Lemma \ref{lem:stabaux} with $s_0 = s$, $u_1 = \phi^\eps(u_0)$ and $u_2 = \phi^{\eps'}(u_0)$. 
		We have by integrations by parts that
		\[\begin{split}
			& \<S^\eps-S^{\eps'},R_{1,S}-R_{2,S}\>+  \<\Lambda^{s+1}\nabla \(S^\eps-S^{\eps'}\),\Lambda^{s+1}\nabla \(R_{1,S}-R_{2,S}\)\> 
			\\
			&\quad\leq c|\eps-\eps'|\|S^\eps-S^{\eps'}\|_{H^{s+2}}\|S^\eps\|_{H^{s+4}},\\
		\end{split}\]
		and
		\[\begin{split}
			& \RE \<\Lambda^{s}\(A^\eps-A^{\eps'}\),\Lambda^{s}\(R_{1,A}-R_{2,A}\)\>
			\leq c|\eps-\eps'|\|A^\eps-A^{\eps'}\|_{H^{s}}\|A^\eps\|_{H^{s+2}}\\
			&\quad+c\|A^\eps-A^{\eps'}\|_{H^{s}}^2\|S^{\eps'}\|_{H^{s+2}}+c\|A^\eps\|_{H^{s}}\|A^\eps-A^{\eps'}\|_{H^{s}}\|S^\eps-S^{\eps'}\|_{H^{s+2}}\\
			&\quad + c|\eps-\eps'|\|A^\eps-A^{\eps'}\|_{H^{s}}\|A^\eps\|_{H^s}\|S^\eps\|_{H^{s+2}}.
		\end{split}\]
		so that
		\[\begin{split}
			&\pa_t\|\phi_t^{\eps}(u_0)-\phi_t^{\eps'}(u_0')\|_{s}^2\leq c\|\phi_t^{\eps}(u_0)-\phi_t^{\eps'}(u_0')\|_{s}^2\(\|\phi_t^{\eps}(u_0)\|_{s+1} + \|\phi_t^{\eps'}(u_0')\|_{s}\)\\
			&\quad\quad+ c|\eps-\eps'|\|\phi_t^{\eps}(u_0)-\phi_t^{\eps'}(u_0')\|_{s}(\|\phi_t^{\eps}(u_0)\|_{s+2}+\|\phi_t^{\eps}(u_0)\|_{s}^2).
		\end{split}\]
	 	Gronwall's Lemma ensures that for all $t\in[0,h_1]$
		\be\label{eq:inequSClinear}
			\|\phi_t^{\eps}(u_0)-\phi_t^{\eps'}(u_0')\|_{s}\leq C\(\|u_0-u_0'\|_{s} + |\eps-\eps'|\)
		\ee
		where 
		\[
			C = C(\|\phi^{\eps}(u_0)\|_{L^\infty([0,h_1],\Sigma_{s+2})},\|\phi^{\eps'}(u_0')\|_{L^\infty([0,h_1],\Sigma_{s})})>0.
		\]
		Thus, 
		$
			(\phi_t^\eps(u_0))_{t\in[0,h_1]}
		$
		is a Cauchy sequence of $\eps$ of $L^\infty([0,h_1],\Sigma_{s})$.
		The limit $\phi^0(u_0)$ is solution of \eqref{eq:SClinear} with $\eps = 0$. Uniqueness follows from \eqref{eq:inequSClinear}. We get immediately that Lemma \ref{lem:0} is also true for $\eps = 0$ and  $\phi^0(u_0)\in L^\infty([0,h_1],\Sigma_{s+2})$.

		Let 
		\[
			T_{max} = \sup\{t>0 :\phi^0(u_0)\in L^\infty([0,t];\Sigma_{s+2})\}>0,
		\]
		then, for any $0<T<T_{max}$, $\phi^0(u_0)\in L^\infty([0,T];\Sigma_{s+2})$.  
		Let us define 
		$\widetilde T = h_1(2M_{s}^0)$ (see Lemma \ref{lem:0} and \eqref{bound}), $C = C(M_{s+2}^0, 2M_{s}^0)$ (see inequality \eqref{eq:inequSClinear}) and $N$ the smallest $n\in \NN$ such that
		\[
			n\widetilde T\geq T.
		\]
		Let $\eps_0>0$ be such that $\eps_0\sum_{j=1}^NC^j\leq M_s^0$ and $\eps\in(0,\eps_0]$.
		By inequality \eqref{eq:inequSClinear} and Lemma \ref{lem:0}, we obtain by induction on $0\leq k\leq N$ that
		\[
			\|\phi_t^{\eps}(u_0)\|_{s}\leq \|\phi_t^{0}(u_0)\|_{s} + \|\phi_t^{0}(u_0)-\phi_t^{\eps}(u_0)\|_{s}\leq M_s^0 + \eps\sum_{j=1}^kC^j\leq M_s^0\leq 2M_{s}^0
		\]
		for all $t\in[0,k\widetilde T]$. Thus,
		$\phi^\eps(u_0)$ is well-defined on $[0,T]$, belongs to  $L^\infty([0,T];\Sigma_{s})$ and
		\be\label{eq:controluniform0}
			\|\phi^\eps(u_0)\|_{L^\infty([0,T];\Sigma_{s})}\leq 2M_{s}^{0}.
		\ee
		Following the arguments of the proofs of Lemmas \ref{lem:flot1}, \ref{lem:flot23} and \ref{lem:flot4}, we obtain that 
		\[
			\pa_t \|\phi^\eps(u_0)\|_{s+2}^2\leq c\|\phi^\eps(u_0)\|_{s+2}^2\|\phi^\eps(u_0)\|_{s} + \|\phi^\eps(u_0)\|_{s+2}\|\mathcal{V}\|_{H^{s+4}}.
		\]
		Gronwall's lemma ensures that	there is $\widetilde C = \widetilde C(M^{s}_{0})>0$ independent of $\eps$ such that 
		\[
			\|\phi_t^\eps(u_0)\|_{s+2}\leq \exp(t\widetilde C)\(t\|\mathcal{V}\|_{H^{s+4}}+ \|u_0\|_{s+2}\)
		\]
		for all $t\in[0,T]$. Moreover, $\phi^\eps(u_0)$ is well-defined in $L^\infty([0,T],\Sigma_{s+2})$ for any $\eps\in(0,\eps_{max}]$. Then, the same arguments ensure that $\eps\in(0,\eps_{max}]\mapsto \phi^\eps(u_0)$ is continuous in $L^\infty([0,T],\Sigma_{s})$ so that $( \phi^\eps(u_0))_{\eps\in[0,\eps_{max}]}$ is uniformly bounded in $L^\infty([0,T],\Sigma_{s+2})$ and the result follows.
	%
	%
	\subsection{The local error estimates}
	The proof of Lemma \ref{lem:3} given in this section is inspired by \cite{auzinger2015defect} where the two flows case is treated.
	The local error of scheme \eqref{scheme1} is defined by
	\[
		\mathscr{R}(h,u) = \varphi^{1234}_h(u) - \phi_h(u).
	\]

	\subsubsection{Main lemmas}
	Let us give the main ingredients that will be used in the proof of Lemma \ref{lem:3}.
	The balls in $\Sigma_{s_0}$ are denoted by
		\be\label{eq:boule}
		 B_{s_0}(M)= \{u\in \Sigma_{s_0}:\;\|u\|_{s_0}\leq M\}
		 \ee
		 for $s_0\geq0$ and $M>0$.
	 The strategy to get estimates on $\mathscr{R}(h,u)$ is to differentiate $\mathscr{R}$ with respect to $h$. Hence, we will be in need of the following lemma whose proof is postponed to Appendix \ref{sec:lemdiff}.
	\begin{lemma}\label{lem:flotdiff}
			Let $s>d/2+1$ and $M>0$. There exists $h_{8} = h_{8}(M)>0$ such that the following two points hold true.
			\begin{enumerate}[(i)]
				\item \label{lempt11} Let $s_1\geq s$. The functions 
				\[\begin{split}
					&(h,u)\in [0,h_{8}]\times\(B_{s}(M)\cap \Sigma_{s_1+3}\)\mapsto \varphi^1_h(u)\in\Sigma_{s_1},
					\\
					&(h,u)\in [0,h_{8}]\times\(B_{s}(M)\cap \Sigma_{s_1+2}\)\mapsto \varphi^2_h(u)\in\Sigma_{s_1},
					\\
					&(h,u)\in [0,h_{8}]\times\(B_{s}(M)\cap \Sigma_{s_1}\)\mapsto \varphi^3_h(u)\in\Sigma_{s_1},
				\end{split}\]
				are $C^1$-applications.
				\item \label{lempt12} Let $s_2\geq s$ and $M_2>0$. There exists $C_{8} = C_{8}(M,M_2)>0$ such that for any $u\in B_{s}(M)\cap B_{s_2+1}(M_2)$, $h\in[0,h_8]$ and any $u_0\in \Sigma_{s_2}$, we have
				\[\begin{split}
					&\|\pa_2\varphi^1_h(u)\cdot u_0\|_{s_2}\leq \exp\(C_{8}h\)\|u_0\|_{s_2},\\
					&\|\pa_2\varphi^2_h(u)\cdot u_0\|_{s_2}\leq \exp\(C_{8}h\)\|u_0\|_{s_2},\\
					&\|\pa_2\varphi^3_h(u)\cdot u_0\|_{s_2}\leq \exp\(C_{8}h\)\|u_0\|_{s_2},
				\end{split}\]
				and
				\[\begin{split}
					&|\< u_0, D\mathcal{N}_1(\varphi^1_h(u))\cdot u_0\>_{s_2}|\leq C_{8}\|u_0\|^2_{s_2},\\
					&|\< u_0, D\mathcal{N}_2(\varphi^2_h(u))\cdot u_0\>_{s_2}|\leq C_{8}\|u_0\|^2_{s_2},\\
					&|\< u_0, D\mathcal{N}_3(\varphi^3_h(u))\cdot u_0\>_{s_2}|\leq C_{8}\|u_0\|^2_{s_2},
				\end{split}\]
		\end{enumerate}
		where $\<\cdot,\cdot\>_{s_0}$ is defined in \eqref{eq:produitscalaire} and $B_{s_0}(M)$ in \eqref{eq:boule}.
	\end{lemma}
	The following lemma ensures that the object studied in the proof of Lemma \ref{lem:3} are well-defined.	
	\begin{lemma}\label{lem:residu}
		Let $s>d/2+1$ and $M>0$.  There is $h_{9} = h_{9}(M)>0$ such that the following three points are true. Let $u\in\Sigma_{s+7}$ such that $\|u\|_{s+2}\leq M$.
		\begin{enumerate}[(i)]
			\item  We have for all $h\in[0,h_9]$,
			\[
				\varphi^{1234}_h(u),\; \varphi^{234}_h(u),\;\varphi^{34}_h(u) \mbox{ and } \varphi^4_h(u)
			\]
			are well-defined, belong to $L^\infty([0,h_{9}],\Sigma_{s+7})$ and satisfy
			\[\begin{split}
				\max\(\|\varphi^4_h(u)\|_{s+2},\|\varphi^{34}_h(u)\|_{s+2},\|\varphi^{234}_h(u)\|_{s+2}\) \leq 4M.
			\end{split}\]
			\item The application $h\in[0,h_9]\mapsto \mathscr{R}(h,u)\in \Sigma_{s}$ is differentiable,
			\[\begin{split}
				&\pa_h \mathscr{R}(h,u) = \sum_{k = 1}^4\N_k(\varphi^{1234}_h(u))-\N_k(\phi_h(u)) + \mathscr{S}(h,u),\\
				&\mathscr{R}(0,u) = 0.
			\end{split}\]
			where
			\[\begin{split}
		 		\mathscr{S}(h,u)  &= \(\chi_{12} + \chi_{13} + \chi_{14}\)(h,\varphi^{234}_h(u)) \\
				&+ \pa_2\varphi^1(h, \varphi^{234}_h(u))\cdot (\chi_{23} + \chi_{24})(h,\varphi^{34}_h(u))\\
				& + \pa_2\varphi^1(h, \varphi^{234}_h(u))\cdot\pa_2\varphi^2(h, \varphi^{34}_h(u))\cdot\chi_{34}(h,\varphi^4_h(u))
			\end{split}\]
			and $\chi_{ij}(h,v) = \pa_2\varphi^i_h(v)\cdot \N_j(v) - \N_j(\varphi^i_h(v))$ $($see \cite[Section 3]{auzinger2015defect}$)$.
			\item Let $v\in \Sigma_{s+7}$. We have,
			\[\begin{split}
				&\pa_h\chi_{ij}(h,v) =  D\N_i(\varphi^i_h(v))\cdot \chi_{ij}(h,v) + [\N_i,\N_j](\varphi^i_h(v))\\
				& \chi_{ij}(0,v) = 0.
			\end{split}\]
		\end{enumerate}
	\end{lemma}

		The following lemma gives bounds on the commutators. 
	\begin{lemma}\label{lem:commutateurs}
		Let $s>d/2+1$. There is $C>0$ such that for any $u\in\Sigma_{s+2}$ and any $1\leq i<j\leq 4$, we have
		\[
			\|[\N_i,\N_j](u)\|_{s}\leq C\|u\|^2_{s+2}(1 + \|u\|_{s+2}).
		\]
		$C$ does not depend on $\eps\in(0,\eps_{max}]$.
	\end{lemma}
	
	\subsubsection{Proof of Lemma \ref{lem:3}}
	%
	%
		%
		Let $s>d/2+1$ and  $M>0$. Let us define $h_4 = h_4(M) = h_9(M)$. Assume for the moment that $u\in \Sigma_{s+7}$ and $\|u\|_{s+2}\leq M$.
		By Lemmas \ref{lem:flotdiff}, \ref{lem:residu}, \ref{lem:commutateurs} and Gronwall's Lemma, there is $C = C(M)>0$ such that for any $h\in [0,h_{4}]$
		\[\begin{split}
			&\|\chi_{12}(h,\varphi^{234}_h(u))\|_{s}+\|\chi_{13}(h,\varphi^{234}_h(u))\|_{s}+\|\chi_{14}(h,\varphi^{234}_h(u))\|_{s}\leq Ch,
			\\
			&\|\chi_{23}(h,\varphi^{34}_h(u))\|_{s}+\|\chi_{24}(h,\varphi^{34}_h(u))\|_{s}\leq Ch
			\\
			&\|\chi_{34}(h,\varphi^{4}_h(u))\|_{s}\leq Ch.
		\end{split}\]
		Using again Lemmas \ref{lem:flotdiff} and \ref{lem:residu}, we obtain that
		\[
			\|\mathscr{S}(h,u)\|_s\leq Ch.
		\]
		Let us define 
		\[\begin{split}
			&R_{1,S} = -\mathcal{V} + \eps^2\Delta \Pi_1 \phi_h(u),\\
			&R_{2,S} = -\mathcal{V} + \eps^2\Delta \Pi_1 \varphi^{1234}_h(u) + \Pi_1\mathscr{S}(h,u),\\
			&R_{1,A} = \frac{i\eps \Delta \Pi_2 \phi_h(u)}{2}-i\eps \frac{\Delta \Pi_1 \phi_h(u)}{2}\Pi_2\phi_h(u),\\
			&R_{2,A} = \frac{i\eps \Delta \Pi_2 \varphi^{1234}_h(u)}{2}-i\eps \frac{\Delta \Pi_1 \varphi^{1234}_h(u)}{2}\Pi_2\varphi^{1234}_h(u)+ \Pi_2\mathscr{S}(h,u)
		\end{split}\]
		where $\Pi_1$ and $\Pi_2$ are defined in \eqref{eq:proj}.
		Then, Lemma \ref{lem:stabaux} ensures that
		\[
			\pa_t\|\varphi^{1234}_h(u)-\phi_h(u)\|^2_{s}\leq C\|\varphi^{1234}_h(u)-\phi_h(u)\|^2_{s} + C\|\varphi^{1234}_h(u)-\phi_h(u)\|_{s}\|\mathscr{S}(h,u)\|_s.
		\]
		Gronwall's lemma ensures that there is $K_4 = K_4(M)$ such that
		\[
			\|\varphi^{1234}_h(u)-\phi_h(u)\|^2_{s}\leq K_4h^2.
		\]
		Let us insist on the fact that $K_4$ and $h_4$ only depend on $M$. Hence, using the fact that for all $h\in [0,h_4]$, the applications
		\[
			u\in \Sigma_{s+2}\mapsto\phi_h(u)\in \Sigma_{s}	
		\]
		and
		\[
			u\in \Sigma_{s+2}\mapsto\varphi^{1234}_h(u)\in \Sigma_{s}	
		\]
		are continuous (see Lemma \ref{lem:1} and the proof of Lemma \ref{lem:flotdiff}), we get that
		\[
			\|\varphi^{1234}_h(u)-\phi_h(u)\|^2_{s}\leq K_4h^2.
		\]
		holds true for any $u\in \Sigma_{s+2}$ such that $\|u\|_{s+2}\leq M/2$ and the result follows.
	%
	\subsubsection{Proof of Lemma \ref{lem:residu}}
	%
	%
	%
	%
	%
		Let $u\in\Sigma_{s+7}$ such that $\|u\|_{s+2}\leq M$. 
		Let us define
		\be\label{eq:boundh2}
			0< h_9 = h_9(M) :=\min\(h_5(4M),h_6(2M),h_7(M),h_8(4M)\) ,
		\ee
		where $h_5,h_6,h_7$ and $h_8$ are defined by Lemmas \ref{lem:flot1}, \ref{lem:flot23}, \ref{lem:flot4} and \ref{lem:flotdiff}.
		Using these lemmas, we get that for all $h\in[0,h_9]$,
		\[
			\varphi^{1234}_h(u),\; \varphi^{234}_h(u),\;\varphi^{34}_h(u) \mbox{ and } \varphi^4_h(u)
		\]
		are well-defined, belong to $L^\infty([0,h_{9}],\Sigma_{s+7})$ and satisfy
		\[\begin{split}
			\max\(\|\varphi^4_h(u)\|_{s+2},\|\varphi^{34}_h(u)\|_{s+2},\|\varphi^{234}_h(u)\|_{s+2}\) \leq 4M.
		\end{split}\]
		Define
		for $i = 1,2,3,4$, $h\geq 0$ and $u_0\in \Sigma_{s+2}$, the applications
		 \[
		 	\vartheta^i(h,u_0) = (h,\varphi^i_h(u_0))^T\mbox{ and } \Xi(h,u_0) = u_0.
		 \]
		By Lemma \ref{lem:flotdiff}, we obtain that
		 \[
		 	h\in [0,h_{9}]\mapsto  \varphi^{1234}_h(u)\in B_{s}(8M)
		 \]
		 is a $C^1$-application since $ \varphi^{1234}_h(u) = \Xi\circ\vartheta^1\circ \vartheta^2\circ \vartheta^3\circ \vartheta^4(h,u)$.
		We have that
		\[\begin{split}
			 &\pa_h \varphi^{1234}_h(u) 
			 \quad= \N_1\varphi^{1234}_h(u) + \pa_2\varphi^{1}_h(\varphi^{234}_h(u))\cdot\N_2\varphi^{234}_h(u) \\
			 &\qquad+ \pa_2\varphi^{1}_h(\varphi^{234}_h(u))\cdot\pa_2\varphi^2_h(\varphi^{34}_h(u))\cdot \N_3\varphi^{34}_h(u) \\
			 &\qquad+ \pa_2\varphi^{1}_h(\varphi^{234}_h(u))\cdot\pa_2\varphi^2_h(\varphi^{34}_h(u))\cdot\pa_2\varphi^{3}_h(\varphi^4_h(u))\cdot\N_4\varphi^4_h(u),
		\end{split}\]		
		so that
		\[\begin{split}
			&\pa_h \varphi^{1234}_h(u) \\
			 &\quad =\N_1\varphi^{1234}_h(u)+\N_2\varphi^{1234}_h(u) + \N_3\varphi^{1234}_h(u)+\N_4\varphi^{1234}_h(u)\\
			 &\qquad+ \chi_{12}(h,\varphi^{234}_h(u))+ \chi_{13}(h,\varphi^{234}_h(u))+ \chi_{14}(h,\varphi^{234}_h(u))\\
			 &\qquad+\pa_2\varphi^{1}_h(\varphi^{234}_h(u))\cdot \(\chi_{23}(h,\varphi^{34}_h(u)) + \chi_{24}(h,\varphi^{34}_h(u))\)\\
			 &\qquad+ \pa_2\varphi^{1}_h(\varphi^{234}_h(u))\cdot\pa_2\varphi^2_h(\varphi^{34}_h(u))\cdot \chi_{34}(h,\varphi^{4}_h(u)).
		\end{split}\]
		Let us show the last point.  We have for $u_0\in\Sigma_{s+7}$ that
		\[\begin{split}
			 &\pa_{h}\(\pa_2\varphi^i_h(v)\cdot u_0\) = D\N_i(\varphi^i_h(v))\cdot \(\pa_2\varphi^i_h(v)\cdot u_0\),
		\end{split}\]
		so that  
		\[\begin{split}
			\pa_h\chi_{ij}(h,v) &=  D\N_i(\varphi^i_h(v))\cdot \pa_2\varphi^i_h(v)\cdot \N_j(v) - D\N_j(\varphi^i_h(v))\cdot\pa_h\varphi^i_h(v)\\
			& = D\N_i(\varphi^i_h(v))\cdot\chi_{ij}(h,v) + [\N_i,\N_j](\varphi^i_h(v)).
		\end{split}\]
	%
	%
	
	\subsubsection{Proof of Lemma \ref{lem:commutateurs}}
	%
		Let us consider 
		\[
			u = \(\begin{array}{c}S\\A\end{array}\) \mbox{ and }u_0 = \(\begin{array}{c}S_0\\A_0\end{array}\).
		\]
		We have
		\[\begin{split}
			&D\N_1(u)\cdot u_0 = \(\begin{array}{c}
			-\nabla S\cdot \nabla S_0\\-\nabla S\cdot \nabla A_0 - A_0\frac{\Delta S}{2}-\nabla S_0\cdot \nabla A - A\frac{\Delta S_0}{2} + i\frac{\Delta A_0}{2}
			\end{array}\),\\
			&D\N_2(u)\cdot u_0 = \N_2u_0=\(\begin{array}{c}
			0\\i(\eps-1)\frac{\Delta A_0}{2}
			\end{array}\),\\
			&D\N_3(u)\cdot u_0 = 0,\\
			&D\N_4(u)\cdot u_0 = \(\begin{array}{c}
			\eps^2\Delta S_0\\
			-i\eps\(A_0\Delta S + A\Delta S_0\)
			\end{array}\),
		\end{split}\]
		so that, $[\N_1,\N_3](u) = 0$, $[\N_2,\N_3](u) = 0$, $[\N_3,\N_4](u) = 0$ and
		\[\begin{split}
			&[\N_1,\N_2](u) = D\N_1(u)\cdot\N_2(u) - D\N_2(u)\cdot\N_1(u)\\
			%
			%
			&\quad=\frac{i(\eps-1)}{2}\(\begin{array}{c}
			0\\ \nabla \Delta S\cdot \nabla A + A\frac{\Delta^2 S}{2} + 2\sum_{k=1}^d\nabla \pa_k S\cdot\nabla \pa_k A  + \pa_kA\frac{\Delta \pa_kS}{2}
			\end{array}\).
		\end{split}\]
		We obtain
		\[\begin{split}
			&\|[\N_1,\N_2](u)\|_{s} %
			\quad\leq C\|u\|_{s+2}^2.
		\end{split}\]
		We also have 
		\[\begin{split}
			&[\N_1,\N_4](u) = D\N_1(u)\cdot\N_4(u) - D\N_4(u)\cdot\N_1(u)\\
			%
			%
			%
			&\quad= \(\begin{array}{c}
			\eps^2\sum_{k=1}^d \nabla \pa_kS \cdot \nabla \pa_k S\\ 
			 (\eps-\eps^2)\(\nabla\Delta S\cdot \nabla A 
			+A\frac{\Delta^2 S}{2} \)
			-i\eps A \sum_{k=1}^d \nabla \pa_kS\cdot \nabla \pa_kS 
			\end{array}\),\\
		\end{split}\]
		and
		\[\begin{split}
			&\|[\N_1,\N_4](u)\|_{s}
			%
			%
			\leq \eps C\|u\|_{s+2}^2\(1 + \|u\|_{s+2}\).
		\end{split}\]
		We also get
		\[\begin{split}
			&[\N_2,\N_4](u) = D\N_2(u)\cdot\N_4(u) - D\N_4(u)\cdot\N_2(u)\\
			%
			%
			%
			&\quad= \frac{\eps(\eps-1)}{2}\(\begin{array}{c}
				0\\
				A\Delta^2S+2\nabla A\cdot \nabla \Delta S
			\end{array}\),
		\end{split}\]
		so that
		\[
			\|[\N_2,\N_4](u)\|_{s}\leq \eps C\|u\|^{2}_{s+2}\ ,
		\]
		and the result follows.
		%
	%

	\section{Numerical experiments}\label{sec:num}
	In this part, we illustrate the behavior of the schemes \eqref{scheme1} and \eqref{scheme2} introduced in Section \ref{sec:schemeconstr}. 
	We restrict ourselves to the one-dimensional periodic setting in which the equations studied remain unchanged and a Fourier spectral discretization can be used. Note that eikonal equation \eqref{eq:SCflot1phase} is solved using the method of characteristics and an interpolation method based on a direct discrete Fourier series evaluation. Many other methods are available to solve this equation. Let us mention in particular \cite{MR3588729,MR3299095} where these questions are discussed in the context of advection equations. 
	
	We consider the following initial data:
	\be\label{eq:initialdata}
	\begin{split}
		&A_0(x) = \sin(x),\quad S_0(x) ={\sin(x)}/{2},\\
	        &\pe(0,\cdot) = A_0(\cdot)e^{iS_0(\cdot)/\eps},
	\end{split}
	\ee
	and the potential
	\[
		\mathcal{V}(x) = \frac{\sin(x)}{1+\cos(x)^2}
	\]
	where $x\in\mathbb{T} = \RR/2\pi\ZZ$, for which caustics appear numerically at time $T_c = 0.8$. In our simulations, the semiclassical parameter $\eps$ varies from $1$ to $2^{-10}$.
	
	The numerical solutions $(S^\eps, A^\eps)$, resp. $\Psi^\eps$, are compared to corresponding reference solutions $(S^\eps_{ref}, A^\eps_{ref})$, resp. $\Psi^\eps_{ref}$, which, in the absence of analytical solutions, are respectively obtained thanks to our second order splitting method \eqref{scheme2} and thanks to a splitting scheme of order $4$ for \eqref{eq:GPElinear} (see \cite{Yoshida}), with very small time and space steps. More precisely, to compute $(S^\eps_{ref}, A^\eps_{ref})$, we have taken $N_x = 2^{8}$ and $h = 2^{-13}T_f$, and to compute $\Psi^\eps_{ref}$, in order to fit with the constraints on the time step and on the space step
	\[
		h \ll \eps \mbox{ and } \Delta x \ll \eps,
	\]
 the space interval $[0,2\pi]$ is discretized with $N_x = 2^{12}$ points and the time step is $h = 2^{-13}T_f$. 	 

	The various errors that are represented in the figures below are defined as follows:
	\[
		\begin{split}
			&err_{\rho^\eps}(T) = \frac{\|\rho_{ref}^\eps(T) - \rho^{\eps}(T)\|_{L^1}}{\|\rho_{ref}^\eps(T)\|_{L^1}},\quad err_{\pe}(T) = \frac{\|\Psi^\eps_{ref}(T)-\pe(T)\|_{L^2}}{\|\psi^\eps_{ref}(T)\|_{L^2}},
		\end{split}
	\]
	and
	\[
		\begin{split}	
			&err_{(\se,\ae)}(T) = \(\frac{\|S^\eps_{ref}(T)-\se(T)\|_{L^2}^2+\|A^\eps_{ref}(T)-\ae(T)\|_{L^2}^2}{\|S^\eps_{ref}(T)\|_{L^2}^2+\|A^\eps_{ref}(T)\|_{L^2}^2}\)^{1/2},
		\end{split}
	\]
	where 
	\[
		\|u\|_{L^1} = \Delta x \sum_{k = 0}^{N_x-1}|u_k|, \quad\|u\|_{L^2} = \sqrt{\Delta x \sum_{k = 0}^{N_x-1}|u_k|^2},
	\] 
	with $\rho_{ref}^\eps(T) = |\Psi^{\eps}_{ref}(T)|^2$ and $\rho^{\eps}(T) = |\ae(T)|^2$. 

	We first study qualitatively the dynamics, in order to guess what is the time of appearance of the caustics.
	Figures \ref{fig:evolDensity} and \ref{fig:evolPhase} represent the density $|A^\eps|^2$ and the phase $S^\eps$ at times $T_f = 0$, $0.3$, $0.6$, $0.8$, $1$  for $\eps = 2^{-4}$. The caustics appear around $t=0.8$. At time $t=1$, oscillations at other scales than those of the phase can be observed in $|A^\eps|^2$ whereas $S^\eps$ ceases to be smooth. These figures are obtained by using our scheme \eqref{scheme2} with $N_x = 2^{8}$ and $N_t = T_f/h = 2^{9}$.
	
	
	Let us now focus on the experiments performed with our first and second-order methods at time $T_f = 0.2$ before the caustics. We start with the first-order scheme \eqref{scheme1}. Figures \ref{fig:Tf1cvgtpsord2-1} and \ref{fig:Tf1cvgtpsord2-2} represent the errors on $\rho^\eps$ and $(\se,\ae)$ {\em w.r.t.}{\@} the time step $h$ for a fixed $N_x = 2^7$. Figures \ref{fig:Tf1cvgdxord2-1} and \ref{fig:Tf1cvgdxord2-2} represent the errors {\em w.r.t.}\@ $\Delta x$ for fixed $N_t = h/T_f = 2^{13}$. All these figures illustrate the fact that our scheme is UA with respect to $\eps$, for the quadratic observables as well as for the whole unknown $(\se,\ae)$ itself. Figures \ref{fig:Tf1cvgtpsord2-1} and \ref{fig:Tf1cvgtpsord2-2} show that \eqref{scheme1} is uniformly of order $1$ in time, whereas Figures  \ref{fig:Tf1cvgdxord2-1} and \ref{fig:Tf1cvgdxord2-2} show that the convergence is uniformly spectral in space. 
	
	Figures \ref{fig:Tf1cvgtpsord4-1} to \ref{fig:Tf1cvgdxord4-2} illustrate the behavior of our second-order scheme \eqref{scheme2} at $T_f = 0.2$: here again, it appears that, before the caustics, our method is UA with an order $2$ in time and with spectral in space accuracy.
	
	Finally, let us explore the behavior of the splitting methods after caustics, by observing the error on the density $\rho^\eps$. Figures \ref{fig:Tf6cvgtpsord4} and \ref{fig:Tf6cvgdxord4} present the same simulations as Figures \ref{fig:Tf1cvgtpsord2-1} and \ref{fig:Tf1cvgdxord2-1}, except that the final time is now $T_f = 1$, \textit{i.e.} we illustrate the behaviors of scheme \eqref{scheme2} after the caustics. In that case, it appears that our methods are not UA, neither in $h$, nor in $\Delta x$, with respect to $\eps$.  Notice that, although it is not UA any longer, our scheme \eqref{scheme2} still has second-order accuracy in time and spectral accuracy in space (with $\eps$-dependent constants). Recall that the same scheme written on \eqref{eq:SCv1} would not be usable in the same situation, since $S^\eps$ ceases to be regular for $\eps>0$, after the formation of caustics.

	\appendix
	\section{Proof of Lemma \ref{lem:flotdiff}}\label{sec:lemdiff}
	\subsection{Study of the differentiability of $\varphi^1$.}
	%
	%
	%
	%
		%
		
		The proof of this lemma is divided in several steps. Let us fix $s>d/2+1$ and $M>0$.

		\subsubsection{Notations.}
		For any  Banach spaces $E$ and $F$, we denote $\mathscr{L}(E,F)$ the set of continuous linear maps between $E$ and $F$ endowed with the norm 
		\[
			\|l\|_{\mathscr{L}(E,F)} = \sup\{\|l(x)\|_{F},\; x\in E, \; \|x\|_E\leq 1\}
		\]
		where $\|\cdot\|_E$ and $\|\cdot \|_F$ are the norms of $E$ and $F$. 

		Let us define for $u_0 = (S_0,A_0)$
		\[\begin{split}
			\Theta^1_h\cdot u_0 = \(\begin{array}{c}
			 	\widetilde S^1_h
				\\
				\widetilde A^1_h
			\end{array}\) 
		\end{split}\]
		the solution of
		\[\begin{split}
			 &\pa_{h}\Theta^1_h = D\N_1(\varphi^1_h(u))\cdot \Theta^1_h\\
			 &\Theta^1_0\cdot u_0 = u_0.
		\end{split}\]
		We denote $\Gamma^1_h =\varphi^1_h(u + u_0)-\varphi^1_h(u) - \Theta^1_h\cdot u_0$,
		\[
			\varphi^1_h(u) = \(\begin{array}{c}
				S^1_h\\
				A^1_h
			\end{array}\)\mbox{, }\varphi^1_h(u + u_0) = \(\begin{array}{c}
				\underline{S^1_h}\\
				\underline{A^1_h}
			\end{array}\),
		\]
		$v^1_h = \nabla S^1_h$, $\underline v^1_h = \nabla \underline S^1_h$, $\widetilde v^1_h = \nabla \widetilde S^1_h$, $\omega^1_h = \underline S^1_h-S^1_h-\widetilde S^1_h$ and $B^1_h = \underline A^1_h-A^1_h-\widetilde A^1_h$.
		%
		

		
		%
		%
		
		%
		%
		
		\subsubsection{Definition of $h_{8}$.}
		Lemma \ref{lem:flot1} ensures that for any	
		 $u\in B_{s}(2M)$,
		we have for $h\in [0,h_{5}(2M)]$ that
		\be\begin{split}\label{eq:lemdiffcontrole}
			&\|\varphi^1_h(u)\|_{s}\leq 4M.
		\end{split}\ee
		We denote $h_8(M) = h_5(2M)$.

		Let $s'\geq s$. If moreover, $u\in\Sigma_{s'}$, then we have
		\be\begin{split}\label{eq:lemdiffcontrole2}
			&\|\varphi^1_h(u)\|_{s'}\leq \exp\(C_{5}(2M)h\)\|u\|_{s'}.
		\end{split}\ee

		\subsubsection{Continuity of $\varphi^1$.}
		Let $s'\geq s$, $M'>0$ and $u_1,u_2\in B_{s}(M)\cap B_{s'+1}(M')$.
		
		By \eqref{eq:lemdiffcontrole} and \eqref{eq:lemdiffcontrole2}, we obtain that
		$\varphi^1_h(u_1)$ and $\varphi^1_h(u_2)$ are well-defined on $[0,h_8]$ and satisfy
		\[
			\|\varphi^1_h(u_1)\|_{s'+1}+\|\varphi^1_h(u_2)\|_{s'+1}\leq 2\exp\(C_{5}(2M)h_8\)M'
		\]
		for all $h\in[0,h_8].$
		By Lemma \ref{lem:stabaux} and an integration by parts, we get that there exists $C = C(M,M')>0$ such that for all $h\in[0,h_8]$
		\be\label{eq:continuityflot1}
			\|\varphi^1_h(u_1)- \varphi^1_h(u_2)\|_{s'}\leq C\|u_1-u_2\|_{s'+1}.
		\ee
		Moreover, for fixed $u\in B_{s}(M)\cap B_{s'+1}(M')$, Lemma \ref{lem:flot1} ensures that $h\in[0,h_8]\mapsto \varphi^1_h(u)\in \Sigma_{s'}$ is continuous so that
		\be\label{eq:continuityflot12}
			(h,u)\in[0,h_8]\times \Sigma_{s'+1}\mapsto \varphi^1_h(u)\in \Sigma_{s'}
		\ee
		is also continuous.

		\subsubsection{Well-posedness, continuity and estimates on the norm for $\Theta^1_h$.}
		Let $s_2\geq s$, $M_2>0$, $u\in B_{s}(M)\cap B_{s_2+1}(M_2)$ and $u_0\in\Sigma_{s_2}$. 
		We recall that the function $\Theta^1_h\cdot u_0 =: (\widetilde S^1_h,\widetilde A^1_h)^T$ satisfies
		\[\begin{split}
			&\pa_h \widetilde S^1_h  +\nabla S^1_h\cdot \nabla \widetilde S^1_h = 0\\
			&\pa_h \widetilde A^1_h +\nabla S^1_h\cdot \nabla \widetilde A^1_h 
			+ \frac{\widetilde A^1_h}{2}\Delta S^1_h
			= -\nabla \widetilde S^1_h\cdot \nabla  A^1_h 
			- \frac{A^1_h}{2}\Delta \widetilde S^1_h
			+\frac{i}{2}\Delta \widetilde A^1_h
		\end{split}\]
		and $\Theta^1_0\cdot u_0 = u_0$.
		The existence and uniqueness of $\widetilde S^1_h$ follows for instance from the method of  characteristics.
		We have
		\[
			\pa_h \widetilde v^1_h + \(v^1_h\cdot \nabla\)\widetilde v^1_h = - \( \widetilde v^1_h\cdot \nabla\) v^1_h
		\]
		and Lemma \ref{eq:burger} with $R = - \( \widetilde v^1_h\cdot \nabla\) v^1_h$ gives us  that 
		\[\begin{split}
			\pa_h\|\widetilde v^1_h\|_{H^{s_2+1}}^2
			%
			%
			&\leq C\|\widetilde v^1_h\|_{H^{s_2+1}}^2\|S^1_h\|_{H^{s_2+3}}\leq C\|\widetilde v^1_h\|_{H^{s_2+1}}^2\|\varphi^1_h(u)\|_{s_2+1}.
		\end{split}\]
		We also have
		\[
			\pa_h \|\widetilde S^1_h\|^2_{L^2}\leq C\|\widetilde S^1_h\|_{L^2}\|\widetilde S^1_h\|_{H^{1}}\|S^1_h\|_{W^{1,\infty}}
		\]
		so that
		\[
			\pa_h \|\widetilde S^1_h\|^2_{H^{s_2+2}}\leq C\|\widetilde S^1_h\|^2_{H^{s_2+2}}\|\varphi^1_h(u)\|_{s_2+1}.
		\]
		The existence and uniqueness of $\widetilde A^1_h$ follows from the fact that
		$\widetilde w^1_h = \widetilde A^1_h\exp\(iS^1_h\)$ satisfies
		\[
			i\pa_h \widetilde w^1_h = -\frac{\Delta}{2} \widetilde w^1_h 
			-\(\nabla \widetilde S^1_h\cdot \nabla  A^1_h 	+ \frac{A^1_h}{2}\Delta \widetilde S^1_h\)\exp\(iS^1_h\).
		\]
		Lemma \ref{eq:transport} with 
		$R = -\nabla \widetilde S^1_h\cdot \nabla  A^1_h 
			- \frac{A^1_h}{2}\Delta \widetilde S^1_h
			+\frac{i}{2}\Delta \widetilde A^1_h$
		ensures that
		\[\begin{split}
			\pa_t\| \widetilde A^1_h\|_{H^{s_2}}^2
			%
			%
			&\leq C\| \Theta^1_h\cdot u_0\|_{s_2}^2\| \varphi^1_h(u)\|_{s_2+1} 		
		\end{split}\]
		so that
		\[
			\pa_t\| \Theta^1_h\cdot u_0\|_{s_2}^2\leq C\| \Theta^1_h\cdot u_0\|_{s_2}^2\| \varphi^1_h(u)\|_{s_2+1}.
		\]
		By \eqref{eq:lemdiffcontrole2} and Gronwall's Lemma, there is $C_{8} = C_{8}(M,M_2)>0$ such that for any $h\in[0,h_{8}]$,
		\be\label{eq:partlemdiff}
			\| \Theta^1_h\cdot u_0\|_{s_2}\leq \exp\(C_{8}h\)\|u_0\|_{s_2}.
		\ee
		Using directly the integrations by parts of the proof of Lemmas \ref{eq:burger} and \ref{eq:transport}, we obtain actually that 
		\[
			|\< u_0, D\mathcal{N}_1(\varphi^1_h(u))\cdot u_0\>_{s_2}|\leq C_{8}\|u_0\|^2_{s_2},
		\]
		for all $u_0\in \Sigma_{s_2}$.

		\subsubsection{Differentiability of $\varphi^1$.}
		By Lemma \ref{lem:flot1} and equations \eqref{eq:SCflot1}, the application 
		\[
			h\in[0,h_8]\mapsto \varphi^1_h(u)\in\Sigma_{s_1}
		\]
		 is differentiable in $h$ for any $u\in B_{s}(M)\cap \Sigma_{s_1+2}$.

		Let us prove that $\varphi^1_h$ is differentiable in $u$ and that $\Theta^1_h$ is its derivative.

		Let $M_1>0$ and $u,u_0\in B_{s}(M)\cap B_{s_1+2}(M_1)$. We have that
		$u,u + u_0\in B_{s}(2M)\cap B_{s_1+2}(2M_1)$.
		By \eqref{eq:lemdiffcontrole}  and \eqref{eq:lemdiffcontrole2}, we obtain that for all $h\in[0,h_{8}]$,
		\[\begin{split}
			&\|\varphi^1_h(u)\|_{s_1+2} + \|\varphi^1_h(u+u_0)\|_{s_1+2}\leq 4\exp\(C_{5}(2M)h\)M_1.
		\end{split}\]
		We have
		\[\begin{split}
			\pa_h \nabla\omega^1_h
			%
			%
			= -(v^1_h+\widetilde v^1_h)\cdot\nabla(\nabla\omega^1_h)
			- \(\nabla\omega^1_h\cdot \nabla\)\underline{v^1_h}
			- \( \widetilde v^1_h\cdot \nabla\)\widetilde v^1_h.
		\end{split}\]
		By Lemma \ref{eq:burger}, we obtain taking $v_1 = v^1_h+\widetilde v^1_h$ and
		\[
			R = - \(\nabla \omega^1_h\cdot \nabla\)\underline{v^1_h}
			- \(\widetilde v^1_h\cdot \nabla\)\widetilde v^1_h
		\]
		that
		\[\begin{split}
			\pa_t\|\nabla\omega^1_h\|_{H^{s_1+1}}^2
			&\leq C\|\nabla\omega^1_h\|_{H^{s_1+1}}^2\(\|v^1_h\|_{H^{s_1+1}}+\|\widetilde v^1_h\|_{H^{s_1+1}}+\|\underline v^1_h\|_{H^{s_1+2}}\)
			%
			%
			%
			\\
			&\qquad+ C\|\nabla\omega^1_h\|_{H^{s_1+1}}\|\widetilde v^1_h\|_{H^{s_1+2}}^2.
		\end{split}\]
		%
		Moreover, we have
		\[\begin{split}
			\pa_h\omega^1_h
			%
			%
			%
			 = -\frac{1}{2}\(\nabla \omega^1_h\cdot (\widetilde  v^1_h+v^1_h)  
			 +\underline v^1_h\cdot \nabla \omega^1_h
			 +|\widetilde v^1_h|^2
			 \)
		\end{split}\]
		so that
		\[\begin{split}
			&\pa_h\|\omega^1_h\|^2_{H^{s_1+2}}\leq\|\Theta^1_h\cdot u_0\|_{s_1+1}^4
			\\
			&\qquad+ C\|\omega^1_h\|_{H^{s_1+2}}^2\(1+\|\varphi^1_h(u)\|_{s_1}+\|\varphi^1_h(u+u_0)\|_{s_1+1}+\|\Theta^1_h\cdot u_0\|_{s_1}\)
		\end{split}\]
		We also have
		\[\begin{split}
			\pa_h B^1_h 
			=
			-\nabla \underline S^1_h\cdot  \nabla B^1_h 
			-  B^1_h\frac{\Delta  \underline S^1_h}{2} 
			-\nabla \omega^1_h\cdot  \nabla (A^1_h+\widetilde A^1_h) 
			-  (A^1_h+\widetilde A^1_h)\frac{\Delta \omega_h^1}{2} 
			\\
			+i\frac{\Delta B^1_h}{2}
			- \nabla \widetilde S^1_h\cdot  \nabla\widetilde A^1_h - \widetilde A^1_h\frac{\Delta \widetilde S^1_h}{2}
			\\
		\end{split}\]
		and Lemma \ref{eq:transport} ensures taking 
		\[
			R = -\nabla \omega^1_h\cdot  \nabla (A^1_h+\widetilde A^1_h) 
			-  (A^1_h+\widetilde A^1_h)\frac{\Delta \omega_h^1}{2} 
			+i\frac{\Delta B^1_h}{2}
			- \nabla \widetilde S^1_h\cdot  \nabla\widetilde A^1_h - \widetilde A^1_h\frac{\Delta \widetilde S^1_h}{2}
		\]
		that,
		\[\begin{split}
			&\pa_t\|B^1_h\|_{H^{s_1}}^2
			%
			\leq \|\Theta^1_h\cdot u_0\|^4_{s_1+1}
			\\
			&\quad+C\|\Gamma^1_h\|_{s_1}^2
			\(1+\|\varphi^1_h(u+u_0)\|_{s_1}+\|\varphi^1_h(u)\|_{s_1+1}+\|\Theta^1_h\cdot u_0\|_{s_1+1}\)
		\end{split}\]
		and
\[\begin{split}
			&\pa_h\|\Gamma^1_h\|_{s_1}^2
			\leq \|\Theta^1_h\cdot u_0\|^4_{s_1+1}
			\\
			&\quad+C\|\Gamma^1_h\|_{s_1}^2
			\(1+\|\varphi^1_h(u+u_0)\|_{s_1+1}+\|\varphi^1_h(u)\|_{s_1+1}+\|\Theta^1_h\cdot u_0\|_{s_1+1}\).
		\end{split}\]
		By \eqref{eq:partlemdiff} with $s_2 = s_1+1$ and Gronwall's Lemma, we get that there exists $C = C(M,M_1)>0$ such that for all $h\in[h,h_{8}]$,
		\[
			\|\Gamma^1_h\|_{s_1}\leq C\|u_0\|_{s_1+1}^2\leq C\|u_0\|_{s_1+2}^2%
		\]
		%
		%
		We proved that for any $h\in[0,h_{8}]$
		\[
			\varphi^1_h : B_{s}(M)\cap \Sigma_{s_1+2}\to \Sigma_{s_1}
		\]
		is differentiable in $B_{s}(M)\cap \Sigma_{s_1+2}$. 

		\subsubsection{Proof of point \eqref{lempt11}.}
		Let us prove that the application
		\[
			(h,u)\in[0,h_{8}]\times \(B_{s}(M)\cap \Sigma_{s_1+4}\) \mapsto \varphi^1_h(u)\in \Sigma_{s_1}
		\]
		is a $C^1$-function.
		
		Using equations \eqref{eq:SCflot1} and \eqref{eq:continuityflot12}, we get that
		\[
			(h,u)\in[0,h_{8}]\times \(B_{s}(M)\cap \Sigma_{s_1+3}\)\mapsto \varphi^1_h(u)\in \Sigma_{s_1+2}
		\]
		is continuous so that the partial derivative
		\[
			(h,u)\in[0,h_{8}]\times \(B_{s}(M)\cap \Sigma_{s_1+3}\) \mapsto \pa_h\varphi^1_h(u)  = \mathcal{N}_1\varphi^1_h(u) \in \Sigma_{s_1}
		\]
		is also continuous.
		Let us study the continuity of 
		\[
			(h,u)\mapsto \pa_2\varphi^1_h(u).
		\]
		Let $u_1,u_2\in B_{s}(M)\cap B_{s_1+2}(M_1)$. We denote $\varphi^1_h(u_i) = (S^{1,i}_h, A^{1,i}_h)$ and $\pa_2\varphi^1_h(u_i)\cdot u_0 = (\widetilde S^{1,i}_h,  \widetilde A^{1,i}_h)$ for $i=1,2$. We have
		\[\begin{split}
			&\pa_h \(\widetilde S^{1,1}_h-\widetilde S^{1,2}_h\)  
			+\nabla S^{1,1}_h\cdot \nabla \(\widetilde S^{1,1}_h-\widetilde S^{1,2}_h\) =  - \nabla  \( S^{1,1}_h- S^{1,2}_h\)\cdot \nabla \widetilde S^{1,2}_h
		\end{split}\]
		so that
		\[\begin{split}
			 &\pa_h\(\nabla\widetilde S^{1,1}_h-\nabla\widetilde S^{1,2}_h\)  
			+\(\nabla S^{1,1}_h\cdot \nabla \)\(\nabla\widetilde S^{1,1}_h-\nabla\widetilde S^{1,2}_h\)
			\\
			 &\quad=
			-\(\(\nabla\widetilde S^{1,1}_h-\nabla\widetilde S^{1,2}_h\)\cdot\nabla\)\(\nabla S^{1,1}_h
			  + \widetilde S^{1,2}_h \)
			  \\
			  &\qquad
			  - \(\nabla \widetilde S^{1,2}_h\cdot \nabla\)\( \nabla S^{1,1}_h- \nabla S^{1,2}_h\).
		\end{split}\]
		By Lemma \ref{eq:burger} with $v_1 = \nabla S^{1,1}_h$ and
		\[\begin{split}
			R &= -\(\(\nabla\widetilde S^{1,1}_h-\nabla\widetilde S^{1,2}_h\)\cdot\nabla\)\(\nabla S^{1,1}_h
			  + \widetilde S^{1,2}_h \)
			  \\
			  &\quad- \(\nabla \widetilde S^{1,2}_h\cdot \nabla\)\( \nabla S^{1,1}_h- \nabla S^{1,2}_h\),
		\end{split}\]
		which satisfies
		\[\begin{split}
			\|R\|_{H^{s_1+1}}
			&\leq 
			C\|\widetilde S^{1,1}_h-\widetilde S^{1,2}_h\|_{H^{s_1+2}}\(\|S^{1,1}_h\|_{H^{s_1+3}}+\|\widetilde S^{1,2}_h\|_{H^{s_1+3}}\)
			\\
			&\qquad +C\|S^{1,1}_h-S^{1,2}_h\|_{H^{s_1+3}}\|\widetilde S^{1,2}_h\|_{H^{s_1+2}}
		\end{split}\]
		we obtain that
		\[\begin{split}
			\pa_t\|\nabla\widetilde S^{1,1}_h-\nabla\widetilde S^{1,2}_h\|_{H^{s_1+1}}^2
			&\leq C\|\nabla\widetilde S^{1,1}_h-\nabla\widetilde S^{1,2}_h\|_{H^{s_1+1}}^2
			\(\|S^{1,1}_h\|_{H^{s_1+3}}+\|\widetilde S^{1,2}_h\|_{H^{s_1+3}}\)
			\\
			&+C\|\nabla\widetilde S^{1,1}_h-\nabla\widetilde S^{1,2}_h\|_{H^{s_1+1}}\|S^{1,1}_h-S^{1,2}_h\|_{H^{s_1+3}}\|\widetilde S^{1,2}_h\|_{H^{s_1+2}}.
		\end{split}\]
		Moreover, we have 
		\[\begin{split}
			&\pa_t \|\widetilde S^{1,1}_h-\widetilde S^{1,2}_h\|^2_{L^2}\\
			&\;\leq C
			 \|\widetilde S^{1,1}_h-\widetilde S^{1,2}_h\|_{L^2}\(\|\widetilde S^{1,1}_h-\widetilde S^{1,2}_h\|_{H^1}\|S^{1,1}_h\|_{W^{1,\infty}}
			 +\|S^{1,1}_h-S^{1,2}_h\|_{H^1}\|\widetilde S^{1,2}_h\|_{W^{1,\infty}}
			\)
		\end{split}\]
		so that
		\[\begin{split}
			\pa_t\|\widetilde S^{1,1}_h-\widetilde S^{1,2}_h\|_{H^{s_1+2}}^2
			&\leq C\|\widetilde S^{1,1}_h-\widetilde S^{1,2}_h\|_{H^{s_1+2}}^2
			\(\|S^{1,1}_h\|_{H^{s_1+3}}+\|\widetilde S^{1,2}_h\|_{H^{s_1+3}}\)
			\\
			&\qquad+C\|\widetilde S^{1,1}_h-\widetilde S^{1,2}_h\|_{H^{s_1+2}}\|S^{1,1}_h-S^{1,2}_h\|_{H^{s_1+3}}\|\widetilde S^{1,2}_h\|_{H^{s_1+2}}.
		\end{split}\]
		We also have
		\[\begin{split}
			&\pa_h \(\widetilde A^{1,1}_h-\widetilde A^{1,2}_h\)
			 +\nabla S^{1,1}_h\cdot \nabla \(\widetilde A^{1,1}_h-\widetilde A^{1,2}_h\) 
			+ \(\widetilde A^{1,1}_h-\widetilde A^{1,2}_h\)\frac{\Delta S^{1,1}_h}{2}
			\\
			&\quad
			=-\nabla \( S^{1,1}_h-S^{1,2}_h\)\cdot \nabla \widetilde A^{1,2}_h 
			- \frac{\widetilde A^{1,2}_h}{2}\Delta \(S^{1,1}_h-S^{1,2}_h\)
			\\
			&\qquad
			-\nabla \(\widetilde S^{1,1}_h-\widetilde S^{1,2}_h\)\cdot \nabla  A^{1,1}_h 
			- \frac{A^{1,1}_h}{2}\Delta \(\widetilde S^{1,1}_h-\widetilde S^{1,2}_h\)
			\\
			&\qquad
			-\nabla \widetilde S^{1,2}_h\cdot \nabla \(A^{1,1}_h-A^{1,2}_h\) 
			- \frac{\(A^{1,1}_h-A^{1,2}_h\)}{2}\Delta \widetilde S^{1,2}_h
			+\frac{i}{2}\Delta \(\widetilde A^{1,1}_h-\widetilde A^{1,2}_h\).
		\end{split}\]
		Using Lemma \ref{eq:transport} with $v_1 = \nabla S^{1,1}_h$ and
		\[\begin{split}
			&R = -\nabla \( S^{1,1}_h-S^{1,2}_h\)\cdot \nabla \widetilde A^{1,2}_h 
			- \frac{\widetilde A^{1,2}_h}{2}\Delta \(S^{1,1}_h-S^{1,2}_h\)
			\\
			&\qquad
			-\nabla \(\widetilde S^{1,1}_h-\widetilde S^{1,2}_h\)\cdot \nabla  A^{1,1}_h 
			- \frac{A^{1,1}_h}{2}\Delta \(\widetilde S^{1,1}_h-\widetilde S^{1,2}_h\)
			\\
			&\qquad
			-\nabla \widetilde S^{1,2}_h\cdot \nabla \(A^{1,1}_h-A^{1,2}_h\) 
			- \frac{\(A^{1,1}_h-A^{1,2}_h\)}{2}\Delta \widetilde S^{1,2}_h
			+\frac{i}{2}\Delta \(\widetilde A^{1,1}_h-\widetilde A^{1,2}_h\).
		\end{split}\]
		which satisfies
		\[\begin{split}
			&\RE\<\Lambda^{s_1}\(\widetilde A^{1,1}_h-\widetilde A^{1,2}_h\), \Lambda^{s_1}R\>
			\\
			&\quad\leq C\|\widetilde A^{1,1}_h-\widetilde A^{1,2}_h\|_{H^{s_1}}
			\|S^{1,1}_h-S^{1,2}_h\|_{H^{s_1+2}}\|\widetilde A^{1,2}_h\|_{H^{s_1+1}}
			\\&\qquad
			+C\|\widetilde A^{1,1}_h-\widetilde A^{1,2}_h\|_{H^{s_1}}\|\widetilde S^{1,1}_h-\widetilde S^{1,2}_h\|_{H^{s_1+2}}\|A^{1,1}_h\|_{H^{s_1+1}}
			\\&\qquad 
			+C\|\widetilde A^{1,1}_h-\widetilde A^{1,2}_h\|_{H^{s_1}}\|A^{1,1}_h-A^{1,2}_h\|_{H^{s_1+1}}\|\widetilde S^{1,2}_h\|_{H^{s_1+2}}
		\end{split}\]
		we obtain that
		\[\begin{split}
			&\pa_h\|\(\pa_2\varphi^1_h(u_1)-\pa_2\varphi^1_h(u_2)\)\cdot u_0\|_{{s_1}}^2
			\\&\qquad
			\leq C\|\(\pa_2\varphi^1_h(u_1)-\pa_2\varphi^1_h(u_2)\)\cdot u_0\|^2_{s_1}\(\|\varphi^1_h(u_1)\|_{s_1+1} + \|\pa_2\varphi^1_h(u_2)\cdot u_0\|_{s_1+1}\)
			\\&\qquad
			+C\|\(\pa_2\varphi^1_h(u_1)-\pa_2\varphi^1_h(u_2)\)\cdot u_0\|_{s_1}
			\|\varphi^1_h(u_1)-\varphi^1_h(u_2)\|_{s_1+1}\|\pa_2\varphi^1_h(u_2)\cdot u_0\|_{s_1+1}.
		\end{split}\]
		Let us recall that $u_1,u_2\in B_{s}(M)\cap B_{s_1+2}(M_1)$. By \eqref{eq:lemdiffcontrole2}, \eqref{eq:continuityflot1} and \eqref{eq:partlemdiff}  with $s_2 = s_1$ and Gronwall's Lemma, we get that for all $h\in[0,h_{8}]$, 
		\[
			u\in B_{s}(M)\cap\Sigma_{s_1+2}\mapsto  \pa_2\varphi^1_h(u)\in \mathscr{L}(\Sigma_{s_1+2},\Sigma_{s_1})
		\]
		is continuous. Hence, we obtain that
		\[\begin{split}
			&(h,u)\in[0,h_{8}]\times \(B_{s}(M)\cap \Sigma_{s_1+3}\)\mapsto (\pa_2\varphi^1_h(u),\pa_h\varphi^1_h(u))\in \mathscr{L}(\Sigma_{s_1+3},\Sigma_{s_1})\times \Sigma_{s_1}
		\end{split}\]
		is continuous and the result follows.

		\subsection{Study of the differentiability of $\varphi^2$ and $\varphi^3$.}

		Let $u,u_0\in \Sigma_{s}$. Since $\mathcal{N}_2$ is linear, we have that
		\[
			\Theta^2_h\cdot u_0 = \varphi^2_h(u_0),
		\]
		$\varphi^2_h$ is differentiable on $\Sigma_{s}$ and for any $h\geq 0$, 
		\[\begin{split}
			&\| \pa_2\varphi^2_h(u)\cdot u_0\|_{s} = \|\varphi^2_h(u_0)\|_{s}= \| u_0\|_{s},\\
			&|\< u_0, D\mathcal{N}_2(\varphi^2_h(u))\cdot u_0\>_{s}|\leq C\|u_0\|^2_{s}.
		\end{split}\]
		and the result follows.
		We easily prove that $\varphi^3_i$ is differentiable, that for any $h\geq 0$, $\Theta^3_h\cdot u_0 = u_0$, that
		\[
			\| \Theta^3_h\cdot u_0\|_{s} = \| \Theta^3_0\cdot u_0\|_{s}= \| u_0\|_{s}.
		\]
		and
		\[
			|\< \chi, \(D\mathcal{N}_3(\varphi^3_h(u))\cdot \chi\)\>_{s}|\leq C\|\chi\|^2_{s},
		\]
		for all $\chi\in \Sigma_s$.

\bibliographystyle{siam}
\bibliography{bibliographiebibdesk}	
	
	\begin{figure}[p]
	\centering
	\hspace{-4cm}
	\begin{subfigure}[t]{0.67\textwidth}
		\centering
		\includegraphics[height=5.cm,width=\textwidth]{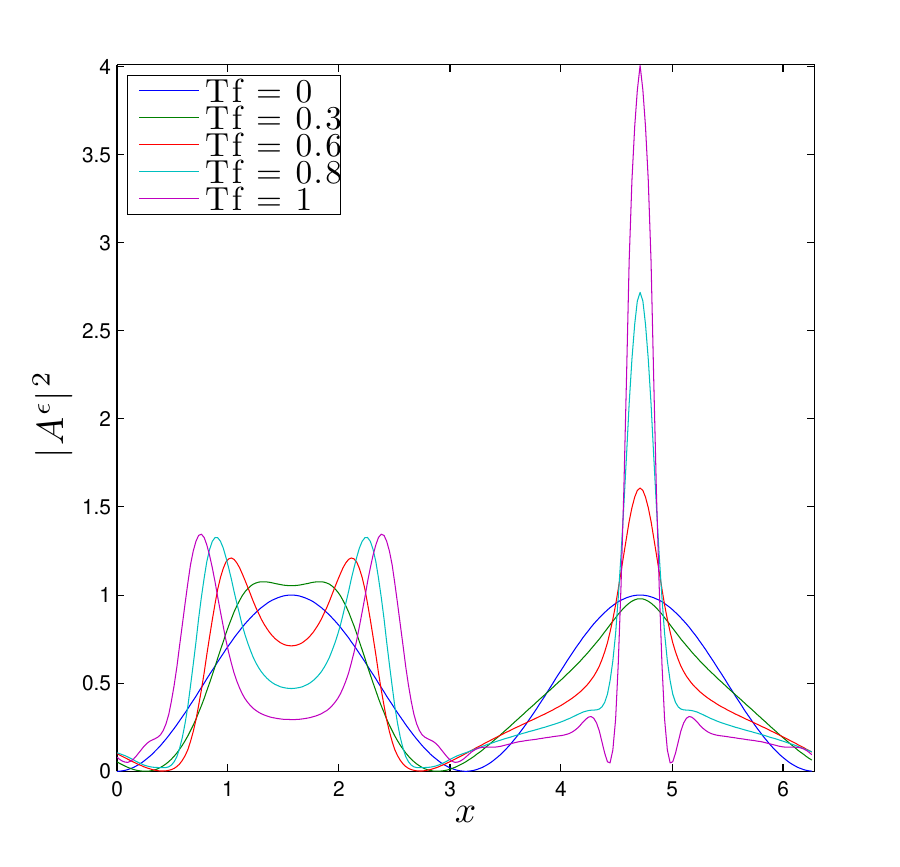}
		\caption{\label{fig:evolDensity} Evolution of the density $|A^\eps|^2$ for $\eps = 2^{-4}$.} 	
	\end{subfigure}
	\hspace{-1cm}
	\begin{subfigure}[t]{0.67\textwidth}
		\centering
		\includegraphics[height=5.cm,width=\textwidth]{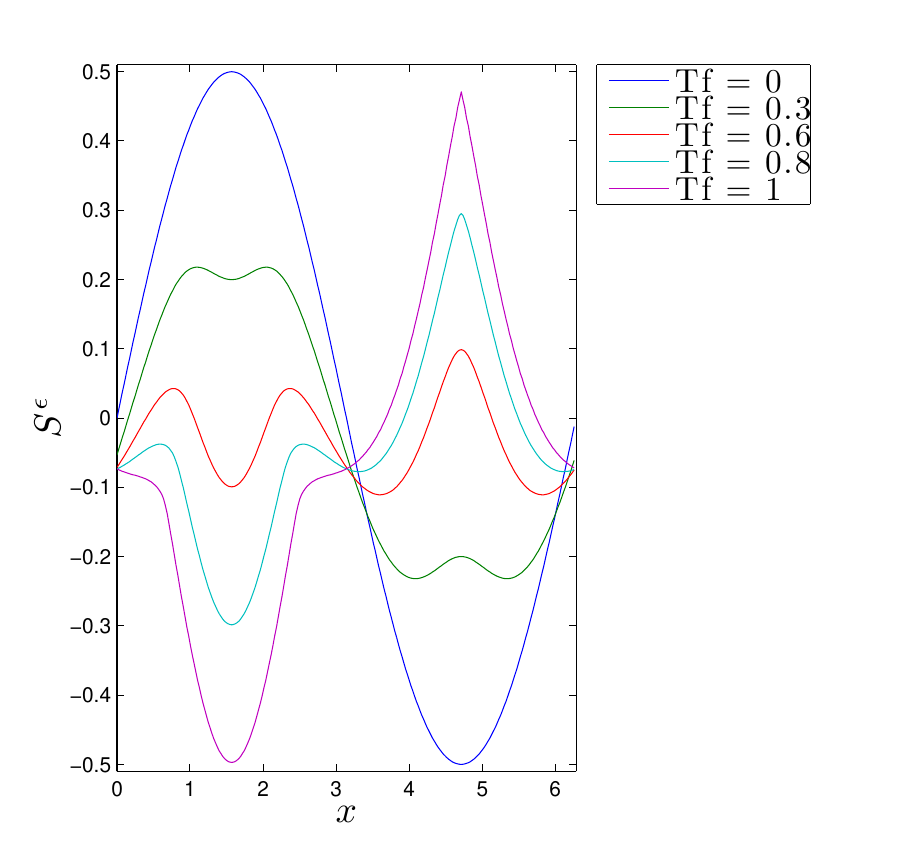}
		\caption{\label{fig:evolPhase} Evolution of the phase $S^\eps$ for $\eps = 2^{-4}$.} 
	\end{subfigure}
	\hspace{-4cm}
	\caption{Evolution of the density and of the phase}
	\end{figure}
	\begin{figure}[p]
	\centering
	\hspace{-4cm}
	\begin{subfigure}[t]{0.63\textwidth}
		\centering
		\includegraphics[height=5.cm,width=\textwidth]{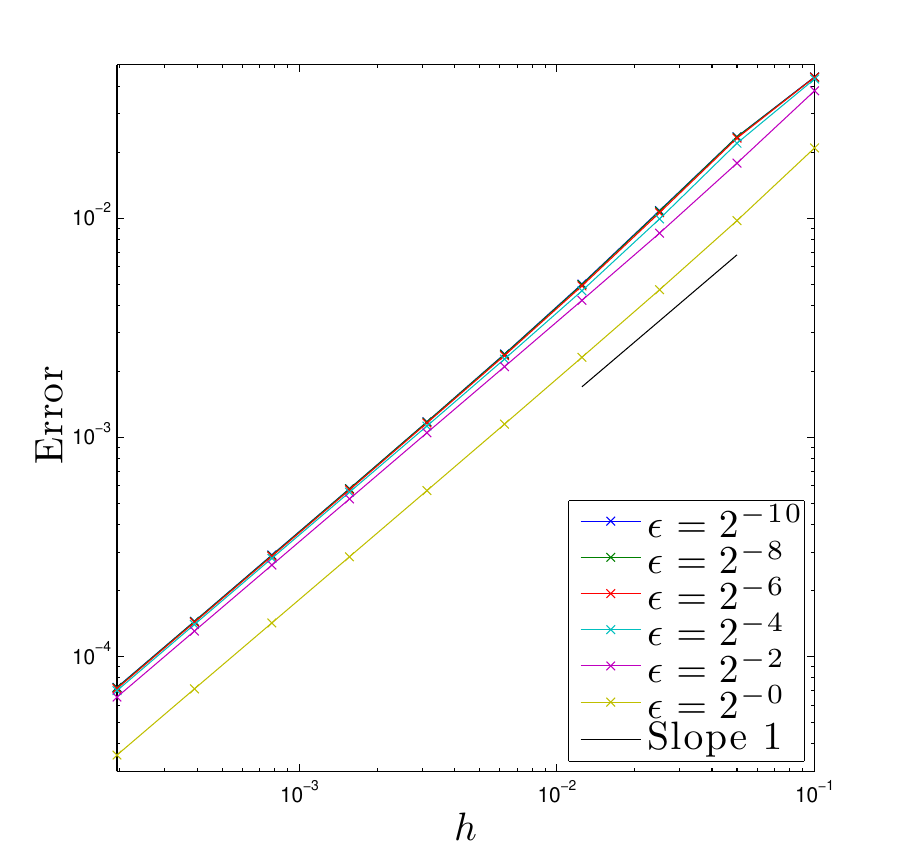}
		\caption{\label{fig:Tf1cvgtpsord2v1} $err_{\rho^\eps} (T_f = 0.2)$ w.r.t $h$, $N_x = 2^{8}$}
	\end{subfigure}
	\hspace{-1cm}
	\begin{subfigure}[t]{0.71\textwidth}
		\centering
		\includegraphics[height=5.cm,width=\textwidth]{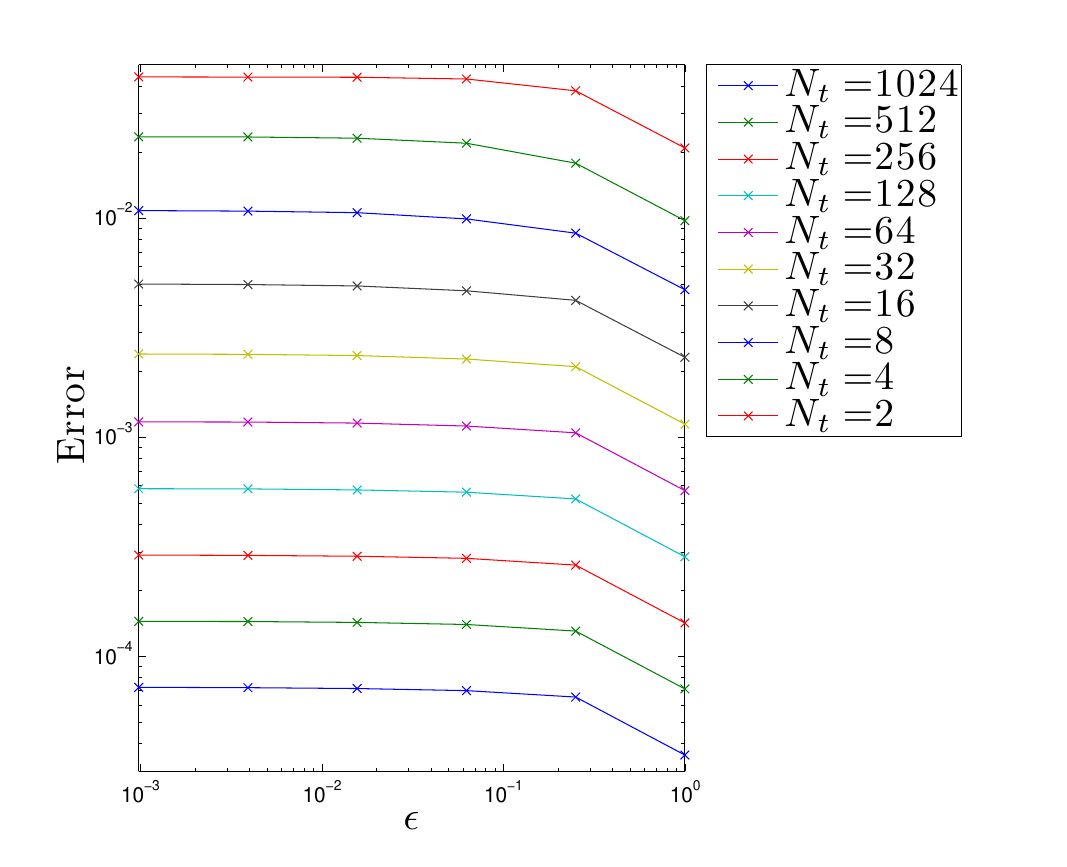}
		\caption{\label{fig:Tf1cvgtpsord2v2} $err_{\rho^\eps} (T_f = 0.2)$ w.r.t $\eps$, $N_x = 2^{8}$} 
	\end{subfigure}
	\hspace{-4cm}
	\caption{\label{fig:Tf1cvgtpsord2-1} Error on the density $\rho^\eps$ for the splitting scheme \eqref{scheme1} of order $1$ before the caustics: dependence on $\eps$ and on $h$.}
		\end{figure}
		\begin{figure}[p]
	\centering
	\hspace{-4cm}
	\begin{subfigure}[t]{0.63\textwidth}
		\centering
		\includegraphics[height=5.cm,width=\textwidth]{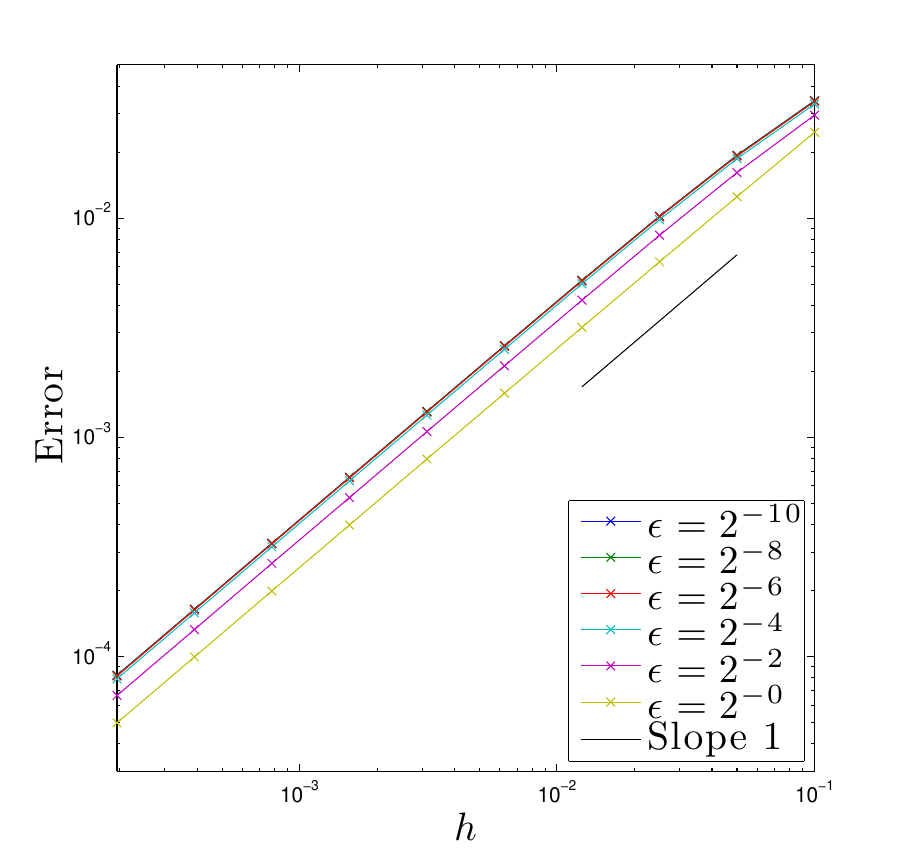}
		\caption{\label{fig:Tf1cvgtpsord2v3} $err_{(\se,\ae)} (T_f = 0.2)$ w.r.t $h$, $N_x = 2^{8}$}
	\end{subfigure}
	\hspace{-1cm}
	\begin{subfigure}[t]{0.71\textwidth}
		\centering
		\includegraphics[height=5.cm,width=\textwidth]{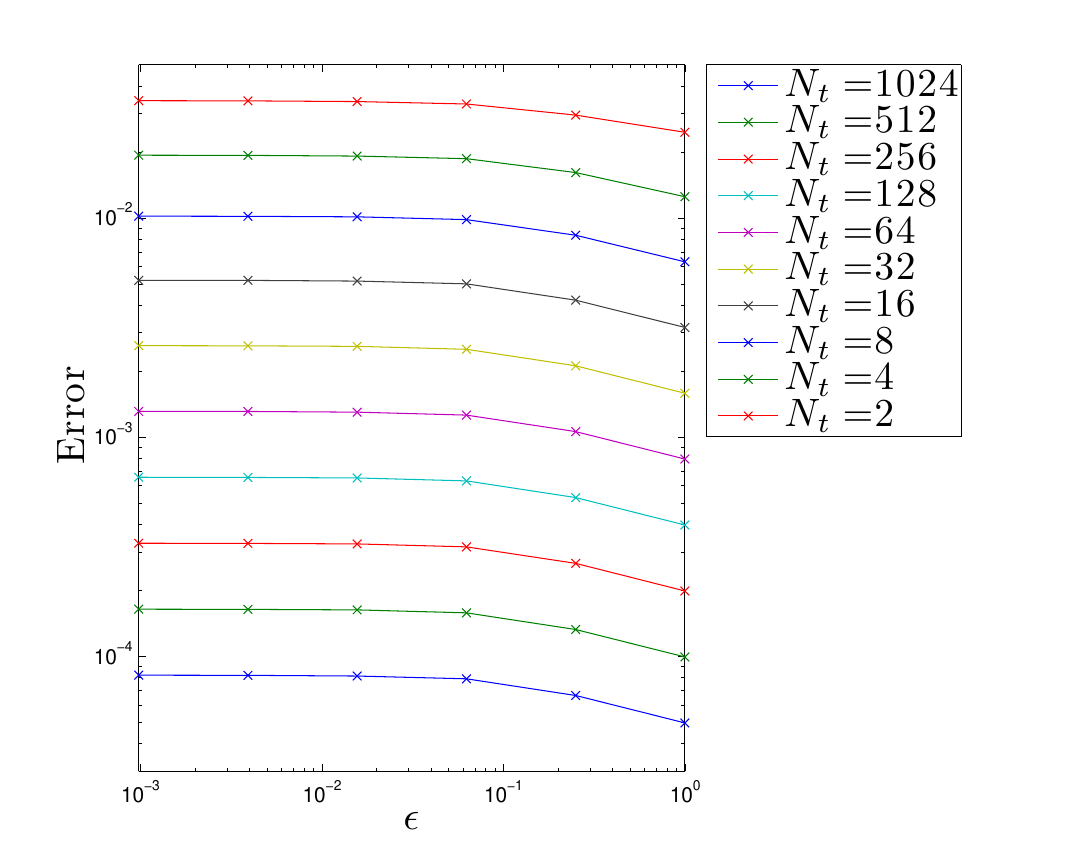}
		\caption{\label{fig:Tf1cvgtpsord2v4} $err_{(\se,\ae)} (T_f = 0.2)$ w.r.t $\eps$, $N_x = 2^{8}$} 
	\end{subfigure}
	\hspace{-4cm}
	\caption{\label{fig:Tf1cvgtpsord2-2} Error on $(\se,\ae)$ for the splitting scheme \eqref{scheme1} of order $1$ before the caustics: dependence on $\eps$ and on $h$.}
		\end{figure}
				\begin{figure}[p]
	\centering
	\hspace{-4cm}
	\begin{subfigure}[t]{0.67\textwidth}
		\centering
		\includegraphics[height=5.cm,width=\textwidth]{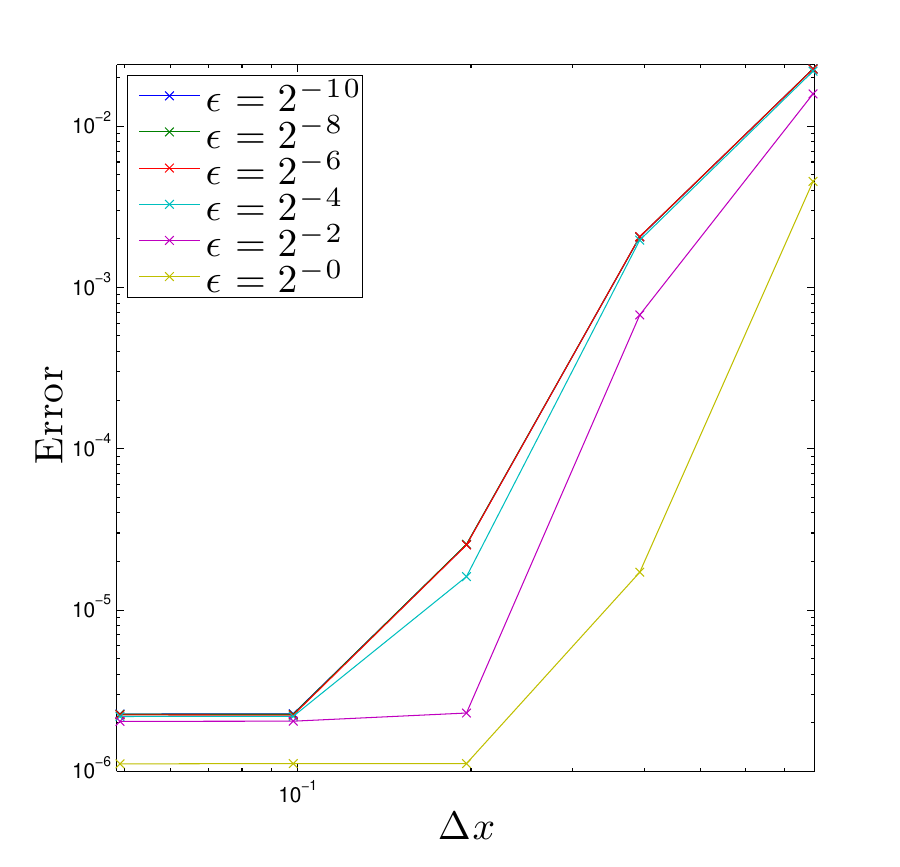}
		\caption{\label{fig:Tf1cvgdxord2v1} $err_{\rho^\eps} (T_f = 0.2)$ w.r.t $\Delta x$, $N_t = 2^{15}$}
	\end{subfigure}
	\hspace{-1cm}
	\begin{subfigure}[t]{0.67\textwidth}
		\centering
		\includegraphics[height=5.cm,width=\textwidth]{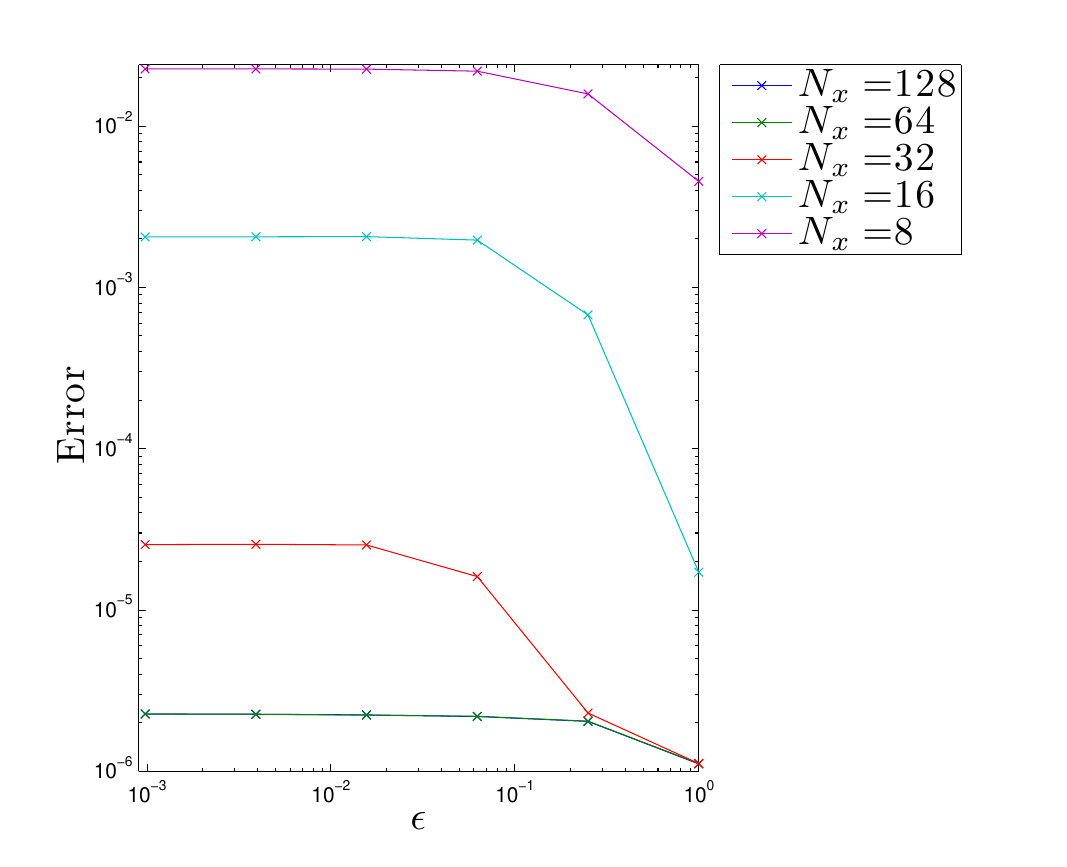}
		\caption{\label{fig:Tf1cvgdxord2v2} $err_{\rho^\eps} (T_f = 0.2)$ w.r.t $\eps$, $N_t = 2^{15}$} 
	\end{subfigure}
	\hspace{-4cm}
		\caption{\label{fig:Tf1cvgdxord2-1} Error on the density $\rho^\eps$ for the splitting scheme \eqref{scheme1} of order $1$ before the caustics: dependence on $\eps$ and on $\Delta x$.}

		\end{figure}
				\begin{figure}[p]
	\centering
	\hspace{-4cm}
	\begin{subfigure}[t]{0.67\textwidth}
		\centering
		\includegraphics[height=5.cm,width=\textwidth]{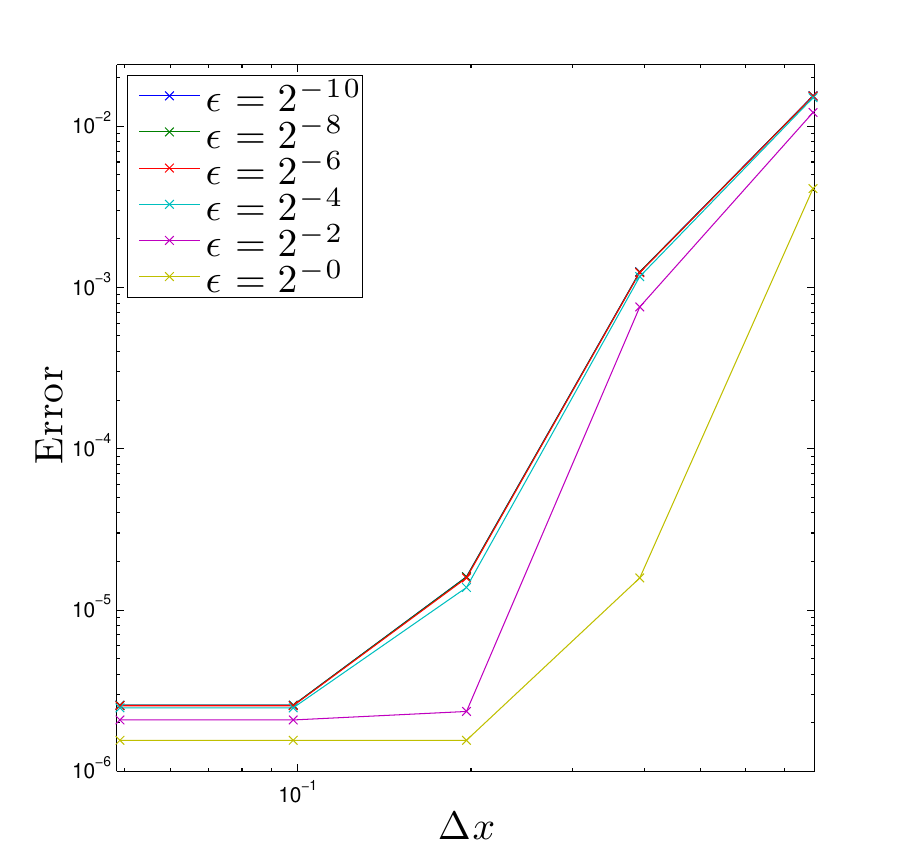}
		\caption{\label{fig:Tf1cvgdxord2v3} $err_{(\se,\ae)} (T_f = 0.2)$ w.r.t $\Delta x$, $N_t = 2^{15}$}
	\end{subfigure}
	\hspace{-1cm}
	\begin{subfigure}[t]{0.67\textwidth}
		\centering
		\includegraphics[height=5.cm,width=\textwidth]{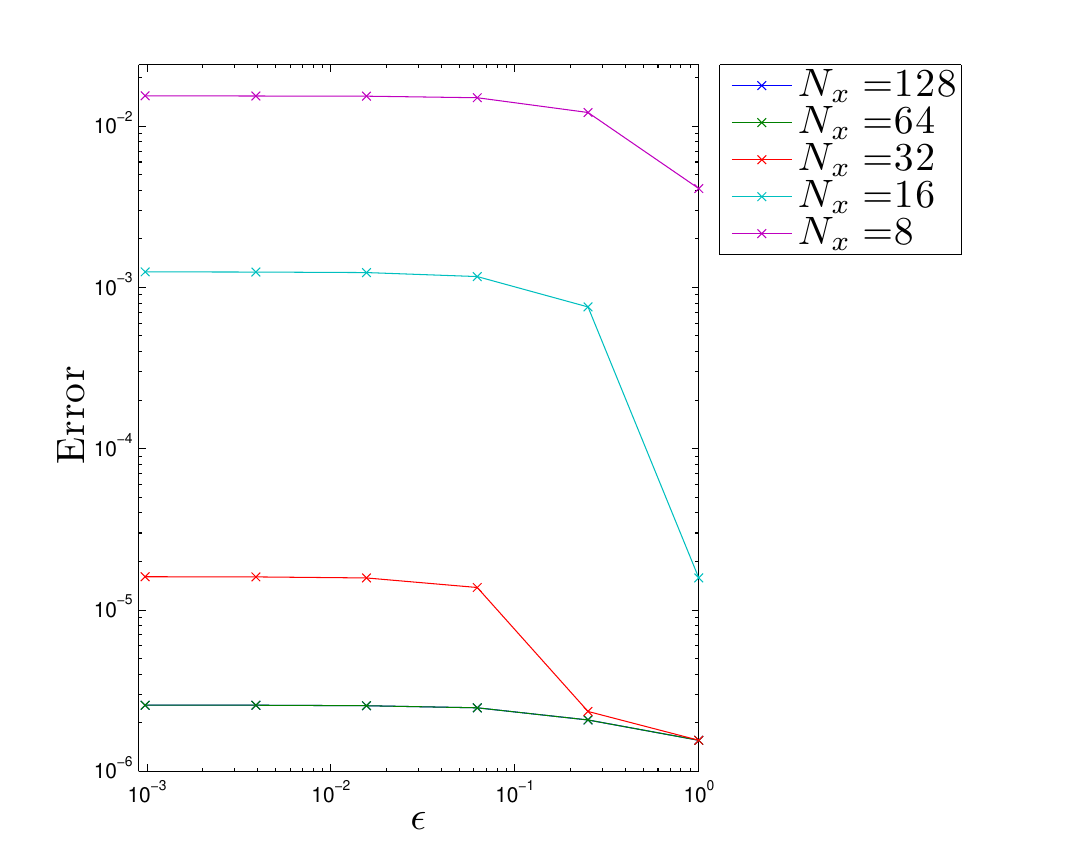}
		\caption{\label{fig:Tf1cvgdxord2v4} $err_{(\se,\ae)} (T_f = 0.2)$ w.r.t $\eps$, $N_t = 2^{15}$} 
	\end{subfigure}
	\hspace{-4cm}
	\caption{\label{fig:Tf1cvgdxord2-2} Error on $(\se,\ae)$ for the splitting scheme \eqref{scheme1} of order $1$ before the caustics: dependence on $\eps$ and on $\Delta x$.}		\end{figure}
	\begin{figure}[p]
	\centering	
		\hspace{-5cm}
		\begin{subfigure}[t]{0.63\textwidth}
		\centering
		\includegraphics[height=5.cm,width=\textwidth]{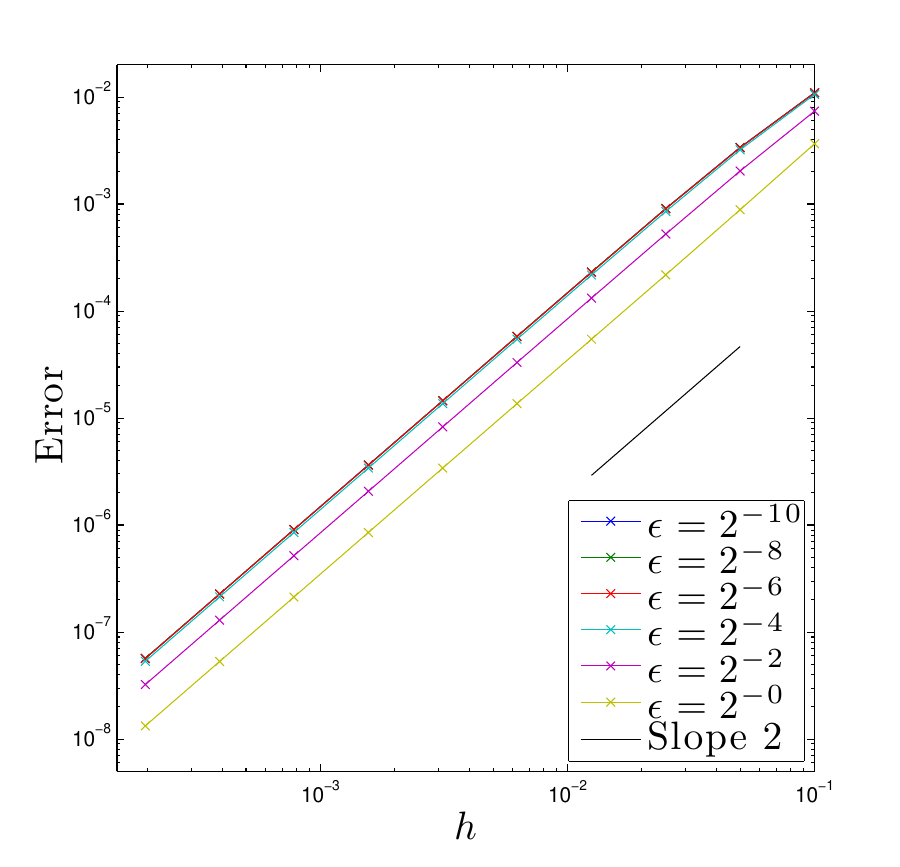}
		\caption{\label{fig:Tf1cvgtpsord4v1} $err_{\rho^\eps} (T_f = 0.2)$ w.r.t $h$, $N_x = 2^8$ } 
		\end{subfigure}
		\hspace{-1cm}
		\begin{subfigure}[t]{0.71\textwidth}
		\centering
		\includegraphics[height=5.cm,width=\textwidth]{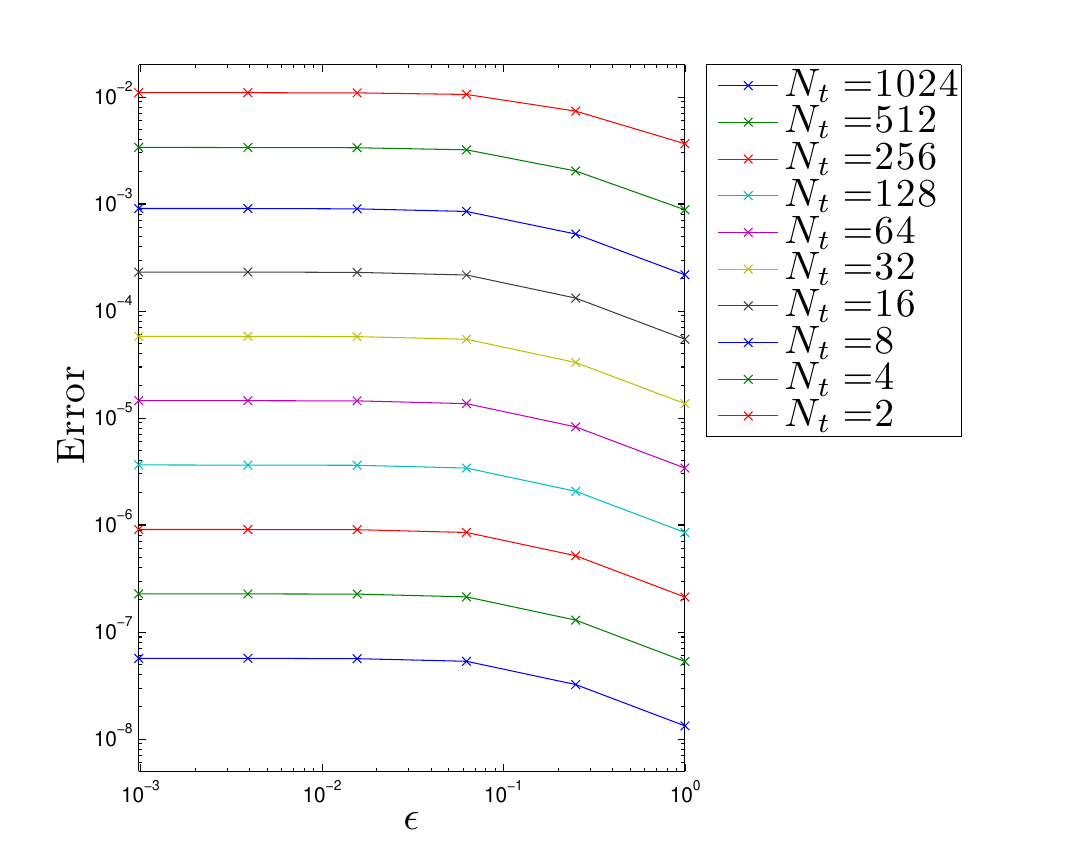}
		\caption{\label{fig:Tf1cvgtpsord4v2} $err_{\rho^\eps} (T_f = 0.2)$ w.r.t $\eps$, $N_x = 2^8$} 
		\end{subfigure}
		\hspace{-6cm}
		\caption{\label{fig:Tf1cvgtpsord4-1}Error on the density $\rho^\eps$ for the splitting scheme \eqref{scheme2} of order $2$ before the caustics: dependence on $\eps$ and on $h$.}

		\end{figure}
		
		\begin{figure}[p]
		\centering	
		\hspace{-5cm}
		\begin{subfigure}[t]{0.63\textwidth}
		\centering
		\includegraphics[height=5.cm,width=\textwidth]{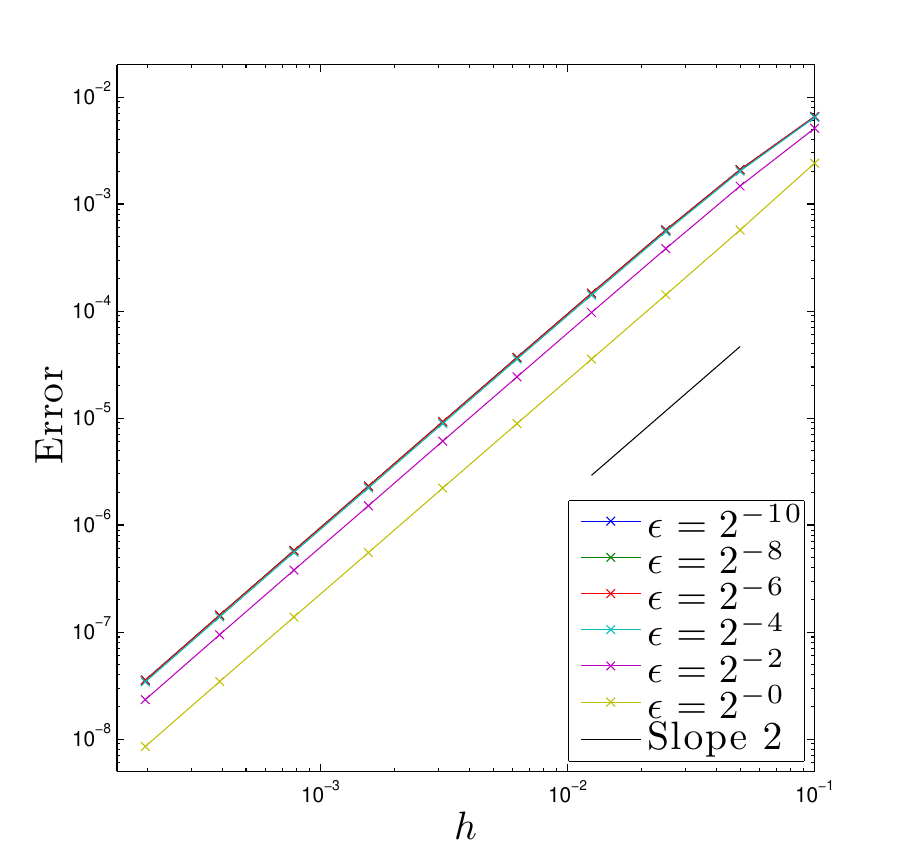}
		\caption{\label{fig:Tf1cvgtpsord4v3} $err_{(\se,\ae)} (T_f = 0.2)$ w.r.t $h$, $N_x = 2^8$ } 
		\end{subfigure}
		\hspace{-1cm}
		\begin{subfigure}[t]{0.71\textwidth}
		\centering
		\includegraphics[height=5.cm,width=\textwidth]{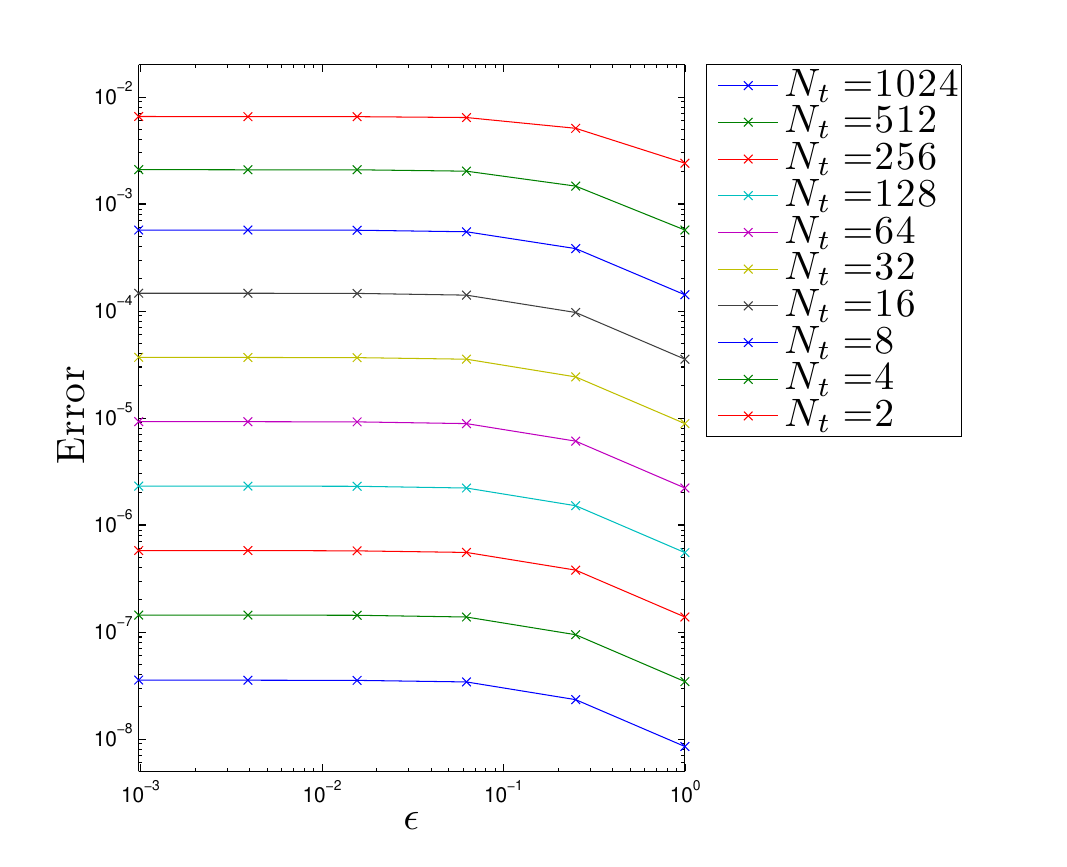}
		\caption{\label{fig:Tf1cvgtpsord4v4} $err_{(\se,\ae)} (T_f = 0.2)$ w.r.t $\eps$, $N_x = 2^8$} 
		\end{subfigure}
		\hspace{-6cm}
		\caption{ \label{fig:Tf1cvgtpsord4-2}Error on $(\se,\ae)$ for the splitting scheme \eqref{scheme2} of order $2$ before the caustics: dependence on $\eps$ and on $h$.}

		\end{figure}
				\begin{figure}[p]
		\centering	
		\hspace{-5cm}
		\begin{subfigure}[t]{0.67\textwidth}
		\centering
		\includegraphics[height=5.cm,width=\textwidth]{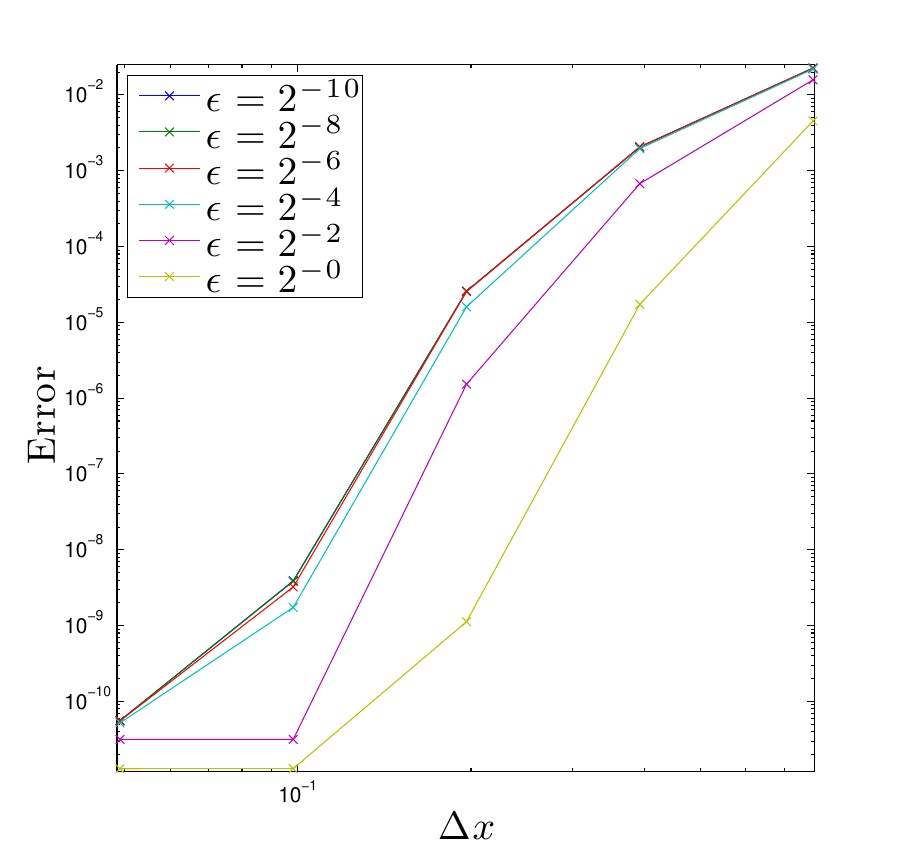}
		\caption{\label{fig:Tf1cvgdxord4v1} $err_{\rho^\eps} (T_f = 0.2)$ w.r.t $\Delta x$, $N_t = 2^{15}$} 
		\end{subfigure}
		\hspace{-1cm}
		\begin{subfigure}[t]{0.67\textwidth}
		\centering
		\includegraphics[height=5.cm,width=\textwidth]{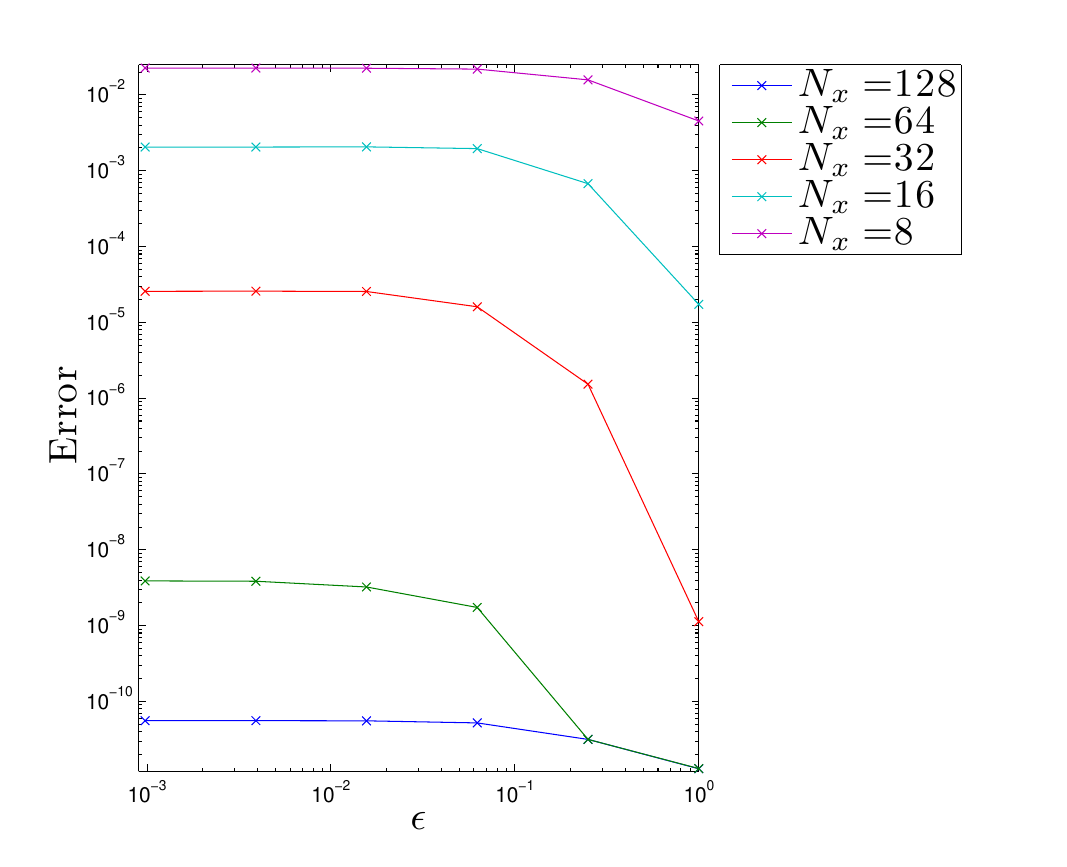}
		\caption{\label{fig:Tf1cvgdxord4v2} $err_{\rho^\eps} (T_f = 0.2)$ w.r.t $\eps$, $N_t = 2^{15}$} 
		\end{subfigure}
		\hspace{-6cm}
		\caption{\label{fig:Tf1cvgdxord4-1} Error on $\rho^\eps$ for the splitting scheme \eqref{scheme2} of order $2$ before the caustics: dependence on $\eps$ and on $\Delta x$.}
		\end{figure}
		
				\begin{figure}[p]
		\centering	
		\hspace{-5cm}
		\begin{subfigure}[t]{0.67\textwidth}
		\centering
		\includegraphics[height=5.cm,width=\textwidth]{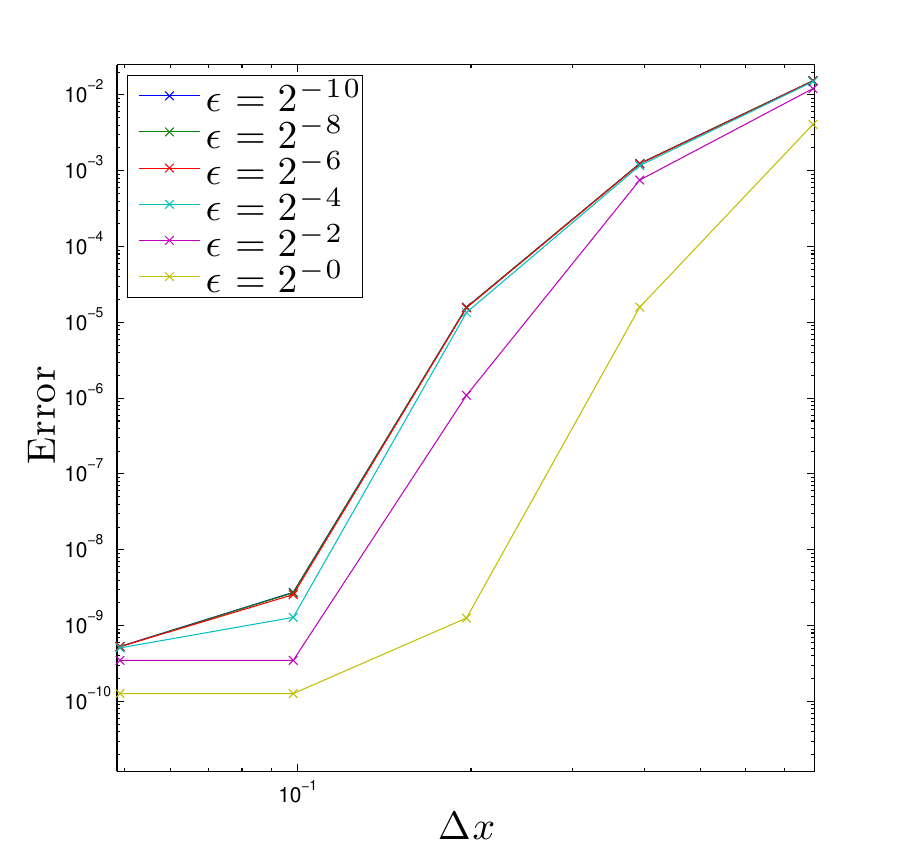}
		\caption{\label{fig:Tf1cvgdxord4v3} $err_{(\se,\ae)} (T_f = 0.2)$ w.r.t $\Delta x$, $N_t = 2^{15}$} 
		\end{subfigure}
		\hspace{-1cm}
		\begin{subfigure}[t]{0.67\textwidth}
		\centering
		\includegraphics[height=5.cm,width=\textwidth]{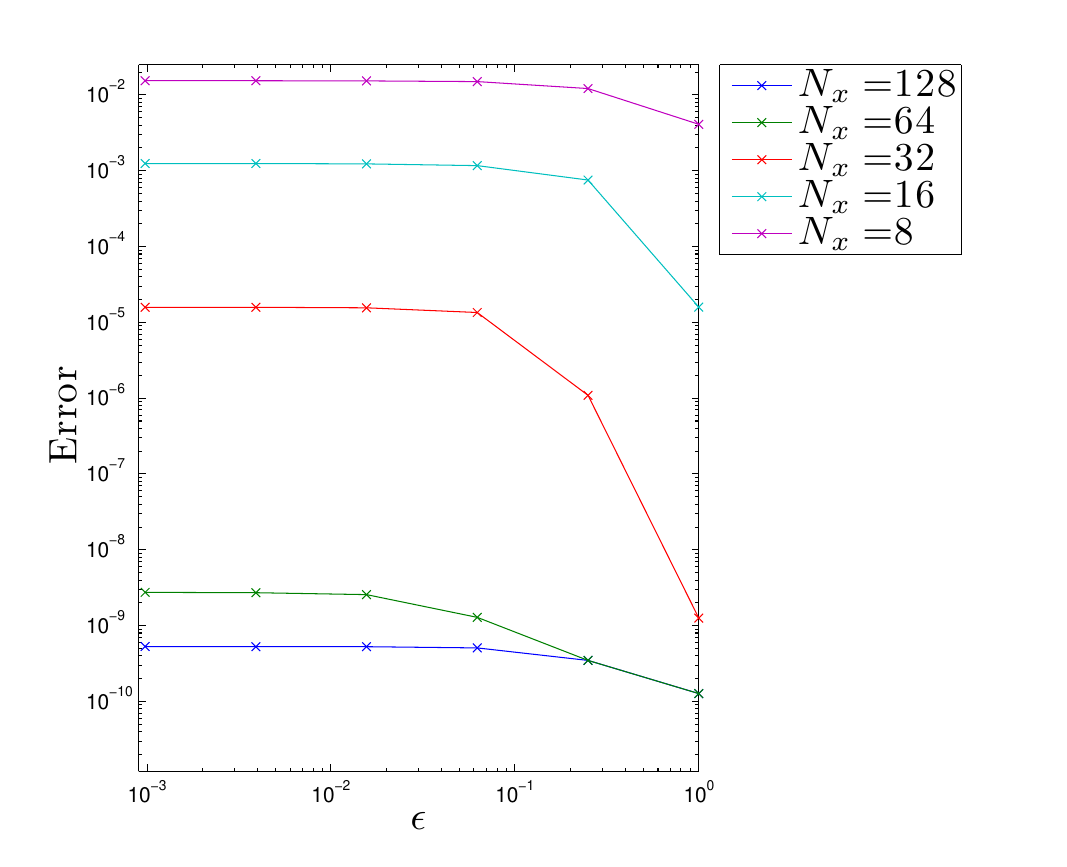}
		\caption{\label{fig:Tf1cvgdxord4v4} $err_{(\se,\ae)} (T_f = 0.2)$ w.r.t $\eps$, $N_t = 2^{15}$} 
		\end{subfigure}
		\hspace{-6cm}
				\caption{\label{fig:Tf1cvgdxord4-2}Error on $(\se,\ae)$ for the splitting scheme \eqref{scheme2} of order $2$ before the caustics: dependence on $\eps$ and on $\Delta x$.}

		\end{figure}
\begin{figure}[p]
	\centering	
		\hspace{-5cm}
		\begin{subfigure}[t]{0.63\textwidth}
		\centering
		\includegraphics[height=5.cm,width=\textwidth]{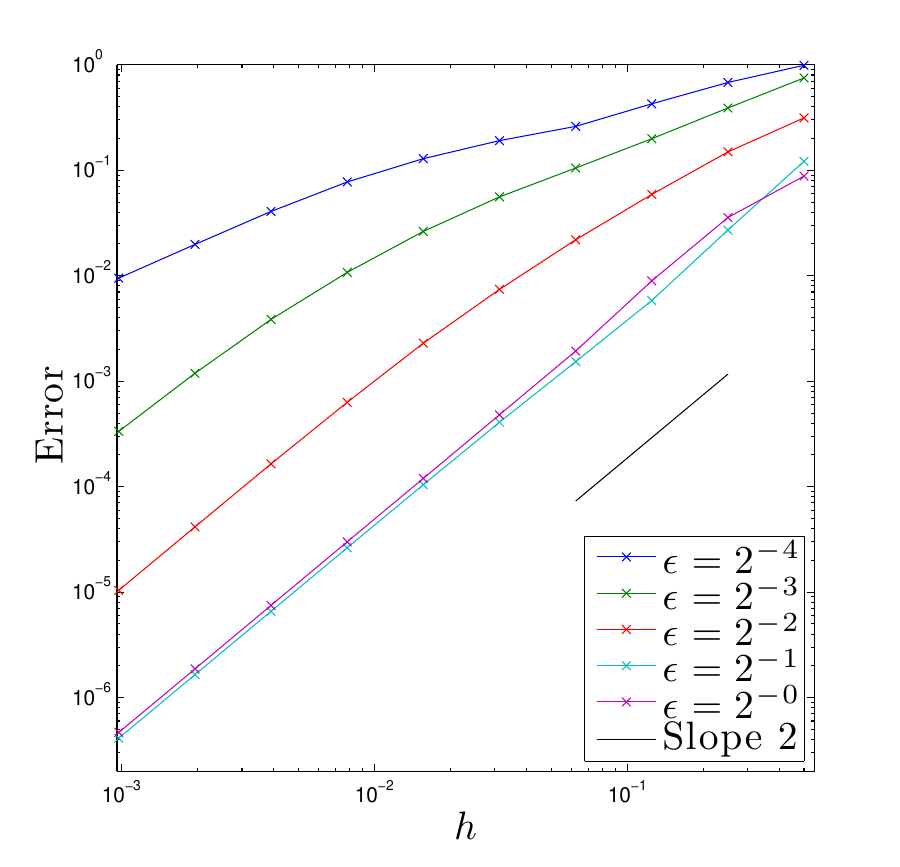}
		\caption{\label{fig:Tf6cvgtpsord4v1} $err_{\rho^\eps} (T_f = 1)$ w.r.t $h$, $N_x = 2^8$ } 
		\end{subfigure}
		\hspace{-1cm}
		\begin{subfigure}[t]{0.71\textwidth}
		\centering
		\includegraphics[height=5.cm,width=\textwidth]{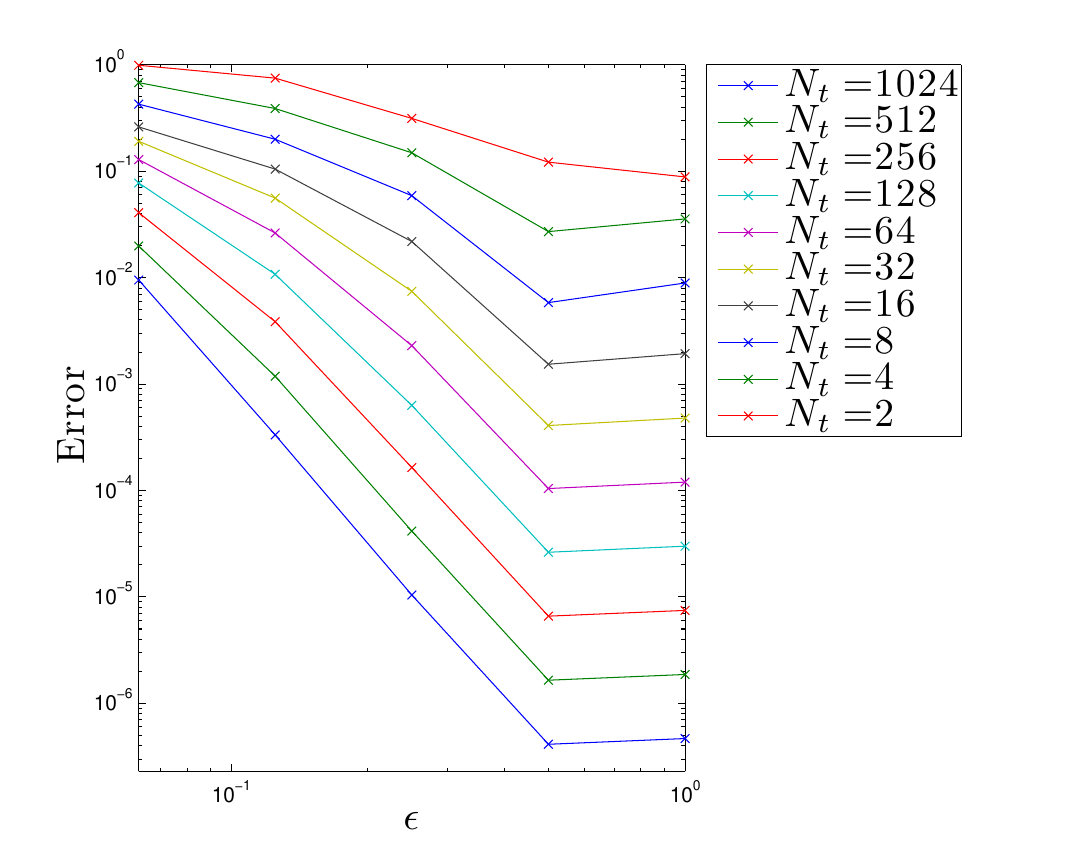}
		\caption{\label{fig:Tf6cvgtpsord4v2} $err_{\rho^\eps} (T_f = 1)$ w.r.t $\eps$, $N_x = 2^8$} 
		\end{subfigure}
		\hspace{-6cm}
		\caption{ \label{fig:Tf6cvgtpsord4}Error on $\rho^\eps$ for the splitting scheme \eqref{scheme2} of order $2$ after the caustics, dependence on $\eps$ and on $h$.}
	
		\end{figure}
		\begin{figure}[p]
		\centering	
		\hspace{-5cm}
		\begin{subfigure}[t]{0.71\textwidth}
		\centering
		\includegraphics[height=5.cm,width=\textwidth]{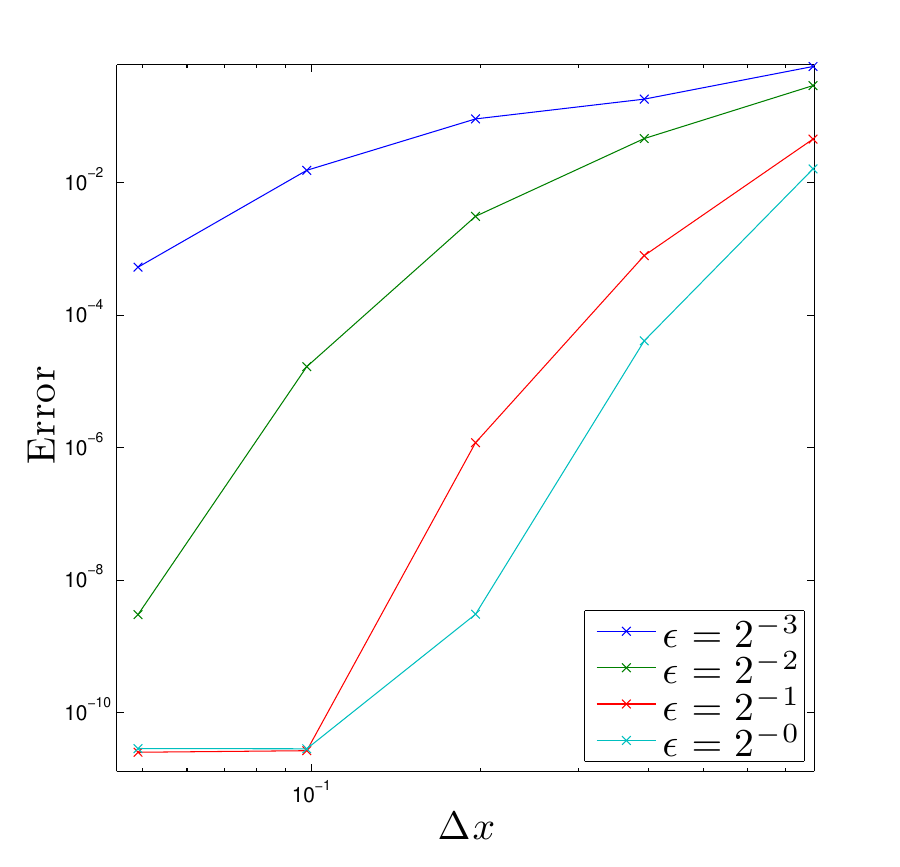}
		\caption{\label{fig:Tf6cvgdxord2v1} $err_{\rho^\eps} (T_f = 1)$ w.r.t $\Delta x$, $N_t = 2^{17}$  } 
		\end{subfigure}
		\hspace{-1cm}
		\begin{subfigure}[t]{0.63\textwidth}
		\centering
		\includegraphics[height=5.cm,width=\textwidth]{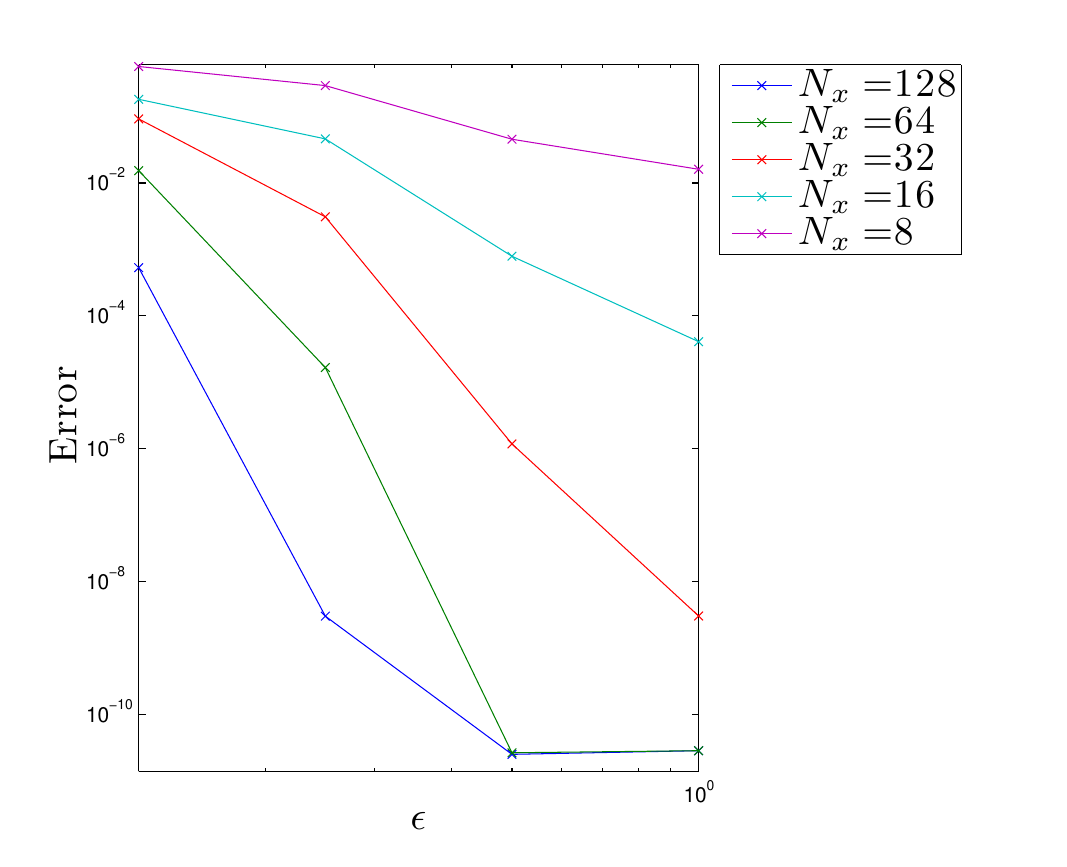}
		\caption{\label{fig:Tf6cvgdxord2v2} $err_{\rho^\eps} (T_f = 1)$ w.r.t $\eps$, $N_t = 2^{17}$ } 
		\end{subfigure}
		\hspace{-6cm}
		\caption{\label{fig:Tf6cvgdxord4}Error on $\rho^\eps$ for the splitting scheme \eqref{scheme2} of order $2$ after the caustics, dependence on $\eps$ and on $\Delta x$.}

		\end{figure}

 \end{document}